\documentclass[11pt,a4paper]{article}
\usepackage[latin1]{inputenc}

\renewcommand{\marginpar}[1]{}

\let\seq\vec

\usepackage{amsmath}
\usepackage{amsfonts}
\usepackage{amssymb}
\usepackage{amsthm}
\usepackage{enumerate}
\usepackage[mathscr]{euscript}
\usepackage{url}

\usepackage[left=25mm,top=25mm,right=25mm]{geometry}

\newtheorem{theorem}{Theorem}
\newtheorem{lemma}[theorem]{Lemma}
\newtheorem{proposition}[theorem]{Proposition}
\newtheorem{corollary}[theorem]{Corollary}
\theoremstyle{definition}
\newtheorem{definition}[theorem]{Definition}

\theoremstyle{remark}
\newtheorem{remark}[theorem]{Remark}
\newtheorem{example}[theorem]{Example}

\newcommand{\omittext}[1]{}

\newcommand{\N}{\mathbb{N}}

\newcommand{\Q}{\mathbb{Q}}
\newcommand{\R}{\mathbb{R}}

\newcommand{\I}{\mathbb{I}}
\renewcommand{\H}{\mathbb{H}}

\newcommand{\Hl}{\mathbb{H}}

\renewcommand{\P}{\mathcal{P}}
\renewcommand{\Pr}{\mathbb{P}}
\newcommand{\Ex}{\mathbb{E}}

\newcommand{\abs}{\mathrm{abs}}

\newcommand{\id}{\mathrm{id}}
\newcommand{\diam}{\mathrm{diam}}
\newcommand{\dom}{\mathrm{dom}}
\newcommand{\rng}{\mathrm{rng}}
\newcommand{\supp}{\mathrm{supp}}
\newcommand{\fto}{\rightarrow}
\newcommand{\mvfto}{\rightrightarrows}
\newcommand{\pfto}{\rightharpoonup}
\newcommand{\mfto}{\rightsquigarrow}
\newcommand{\psfto}{\rightharpoonup}

\newcommand{\clN}{\,\overline{\!N}}
\newcommand{\clB}{\overline{B}}

\newcommand{\calT}{\mathcal{T}}
\newcommand{\calV}{\mathcal{V}}
\newcommand{\calX}{\mathcal{X}}

\newcommand{\opset}{\mathcal{O}}
\newcommand{\clset}{\mathcal{A}}
\newcommand{\cpset}{\mathcal{K}}

\newcommand{\borset}{\mathcal{B}}
\newcommand{\ctsfn}{\mathcal{C}}
\newcommand{\rand}{\mathcal{R}}
\newcommand{\meas}{\mathcal{M}}
\newcommand{\prob}{\mathcal{P}}

\newcommand{\standardtype}[1]{\mathbb{#1}}
\newcommand{\standardfunctor}[1]{\mathcal{#1}}
\newcommand{\usertype}[1]{{#1}}
\renewcommand{\usertype}[1]{\mathbb{#1}}

\newcommand{\tpBl}{\standardtype{B}}
\newcommand{\tpSi}{\standardtype{S}}
\newcommand{\tpNat}{\standardtype{N}}
\newcommand{\tpHl}{\standardtype{H}_<}
\renewcommand{\tpHl}{\standardtype{H}}
\newcommand{\tpIv}{\standardtype{I}}
\newcommand{\tpRe}{\standardtype{R}}
\newcommand{\tpMe}{\standardfunctor{M}}
\newcommand{\tpMon}{\standardfunctor{M}}
\newcommand{\tpPr}{\standardfunctor{P}}
\newcommand{\tpCts}{\standardfunctor{C}}

\newcommand{\tpRV}{\standardfunctor{R}}
\newcommand{\tpOp}{\standardfunctor{O}}
\newcommand{\tpCl}{\standardfunctor{A}}
\newcommand{\tpR}{\usertype{R}}
\newcommand{\tpT}{\usertype{T}}
\newcommand{\tpX}{\usertype{X}}
\newcommand{\tpY}{\usertype{Y}}
\newcommand{\tpZ}{\usertype{Z}}

\let\seq\vec

\newcommand{\Andre}{{Andr\'e}}
\newcommand{\Holder}{{H\"older}}
\newcommand{\Ito}{{It\=o}}
\newcommand{\Levy}{L\'evy}

\newcommand{\chif}{\raisebox{0.44ex}{\ensuremath{\chi}}}

\newcommand{\true}{{\sf{T}}}
\newcommand{\false}{{\sf{F}}}
\newcommand{\indet}{\bot}

\newcommand{\init}{\mathrm{init}}

\DeclareMathOperator{\sign}{\mathrm{sign}}

\begin{document}

\title{Computable Stochastic Processes}
\author{Pieter Collins\\Department of Knowledge Engineering\\Maastricht University\\\texttt{pieter.collins@maastrichtuniversity.nl}}
\date{15 September 2014}
%\date{Final Draft \today}
%\institute
\maketitle

\begin{abstract}
The aim of this paper is to present an elementary computable theory of probability, random variables and stochastic processes.
The probability theory is based on existing approaches using valuations and lower integrals.
Various approaches to random variables are discussed, including the approach based on completions in a Polish space.
We apply the theory to the study of stochastic dynamical systems in discrete-time, and give a brief exposition of the Wiener process as a foundation for stochastic differential equations.
The theory is based within the framework of type-two effectivity, so has an explicit direct link with Turing computation, and is expressed in a system of computable types and operations, so has a clean mathematical description.
\end{abstract}

%\subsubsection*{References} \noindent Edalat~\cite{Edalat1995DSM,Edalat1995DTI,Edalat2009}, \cite{DingWu2005}, \cite{Weihrauch1999}. \cite{AlvarezManillaEdalatSahebDjahromi2000}, \cite{JonesPlotkin1989}, \cite{GoubaultLarrecq2005}.

\section{Introduction}

In this paper, we present a computable theory of probability, random variables and stochastic processes, with the aim of providing a theoretical foundation for the rigorous numerical analysis of discrete-time continuous-state Markov chains and stochastic differential equations.
The first part of the paper provide an exposition of the approach to probability distributions using valuations and the development of integrals of positive lower-semicontinuous and of bounded continuous functions, and on the approach to random variables as limits of almost-everywhere defined continuous partial functions.
In the second part, we show that our approach allows one to very quickly derive computability results for discrete-time stochastic processes.
In the third part, we provide a new construction of the Wiener process in which sample paths are effectively computable, and use this to show that the solutions to stochastic differential equations can be effectively computed.

An early approach to constructive measure theory was developed in~\cite{BishopCheng1972}; see also~\cite{BishopBridges1985}.
The standard approach to a constructive theory of probability measures, as developed in~\cite{JonesPlotkin1989,Edalat1995DTI,SchroderSimpson2006,Escardo2009}, is through \emph{valuations}, which are measures restricted to open sets.
The most straightforward approach to integration is the \emph{Choquet} or \emph{horizontal} integral, a lower integral introduced within the framework of domain theory  in~\cite{Tix1995}; see also~\cite{Konig1995,Lawson2004}.
The lower integral on valuations in the form used here was given in~\cite{Vickers2008}.
Relationships between the constructive and classical approaches were given in ~\cite{Edalat1995DTI}.
Explicit representations of valuations within the framework of type-two effectivity were given in~\cite{Schroder2007}, and representation of probability measures using probabilistic processes were given by~\cite{SchroderSimpson2006}.
In~\cite{Escardo2009}, a language EPCL for nondeterministic and probabilistic computation was given, based on the PCL language of~\cite{Escardo2004}.
%Category~\cite{Giry1982}.
In~\cite{HoyrupRojas2009IC}, a theory of measure was developed for the study of algorithmic randomness.

A constructive theory of measurable functions was also developed in~\cite{BishopCheng1972,BishopBridges1985}.
The theory is developed using abstract \emph{integration spaces}, and the integral is extended from test functions to integrable functions by taking limits.
However, the approach we use here, in which measurable functions are defined as limits of effectively-converging Cauchy sequences of continuous functions was introduced in~\cite{Spitters2003} and further developed in~\cite{Spitters2006,CoquandSpitters2009}.
Random variables over discrete domains were defined in~\cite{Mislove2007}, based on work of~\cite{Varacca2002}.
This was extended to random variables over continuous domains in~\cite{GoubaultLarrecqVaracca2011}, but the construction allows only for continuous random variables, and is overly-restrictive in practice.

To the best of our knowledge there has been relatively little work on constructive and computable approaches to stochastic processes.
An early constructive theory of discrete-time stochastic processes focusing on stopping times was given in~\cite{Chan1972}.
A fairly comprehensive theory though technically advanced theory based on stochastic relations is developed in~\cite{Doberkat2007}; the approach here is considerably simpler.
The monadic properties of the lower integral on valuations, were noted by~\cite{Vickers2011PP}, and of the completion construction by~\cite{OConnorSpitters2010}.
.%A constructive probability theory was developed in~\cite{Chan1974}, and

We use the framework of \emph{type-two effectivity (TTE)}, in which computations are performed by Turing machines working on infinite sequences, as a foundational theory of computability.
We believe that this framework is conceptually simpler for non-specialists than the alternative of using a domain-theoretic framework.
Since in TTE we work entirely in the class of quotients of countably-based (QCB) spaces, which form a cartesian closed category, many of the basic operations can be carried out using simple type-theoretic constructions such as the $\lambda$-calculus.

We assume that the reader has a basic familiarity with classical probability theory~(see e.g.~\cite{Shiryaev1995}) and stochastic processes~(see~\cite{Friedman1975,WatanabeIkeda1981,GikhmanSkorohod2004}).
Much of this article is concerned with giving computational meaning to classical concepts and arguments.
The main difficulty lies in the use of $\sigma$-algebras in classical probability, which have poor computability properties.
Instead, we use only topological constructions, which can usually be effectivised directly.
In particular, we define types of measurable functions as a completion of types of continuous functions.

%\cite{EvansLN,FannjiangLN,LalleyLN}

\section{Computable Analysis}

In the theory of type-two effectivity, computations are performed by Turing machines acting on \emph{sequences} over some alphabet $\Sigma$.
A computation performed by a machine $\meas$ is \emph{valid} on an input $p\in\Sigma^\omega$ if the computation does not halt, and writes infinitely many symbols to the output tape.
A type-two Turing machine therefore performs a computation of a partial function $\eta:\Sigma^\omega \pfto \Sigma^\omega$; we may also consider multi-tape machines computing $\eta:(\Sigma^\omega)^n \pfto (\Sigma^\omega)^m$.
It is straightforward to show that any machine-computable function $\Sigma^\omega \pfto \Sigma^\omega$ is continuous on its domain.

In order to relate Turing computation to functions on mathematical objects, we use \emph{representations} of the underlying sets, which are partial surjective functions $\delta:\Sigma^\omega \psfto \tpX$.
An operation $\tpX\fto \tpY$ is $(\delta_\tpX;\delta_\tpY)$-computable if there is a machine-computable function $\eta:\Sigma^\omega \pfto \Sigma^\omega$ with $\dom(\eta)\supset\dom(\delta_\tpX)$ such that $\delta_\tpY\circ \eta = f\circ\delta_\tpX$ on $\dom(\delta_\tpX)$.
Representations are equivalent if they induce the same computable functions.
If $\tpX$ is a topological space, we say that a representation $\delta$ of $\tpX$ is an \emph{admissible quotient representation} if (i) whenever $f:\tpX\fto \tpY$ is such that $f\circ\delta$ is continuous, then $f$ is continuous, and (ii) whenever $\phi:\Sigma^\omega\pfto \tpX$ is continuous, there exists continuous $\eta:\Sigma^\omega\pfto\Sigma^\omega$ such that $\phi=\delta\circ\eta$.
A computable type is a pair $(\tpX,[\delta])$ where $\tpX$ is a space and $[\delta]$ is an equivalence class of admissible quotient representations of $\tpX$.
A multivalued function $F:\tpX \mvfto \tpY$ is \emph{computably selectable} if there is a machine-computable function $\eta:\Sigma^\omega \pfto \Sigma^\omega$ with $\dom(\eta)\supset\dom(\delta_\tpX)$ such that $\delta_\tpY\circ \eta \in F\circ\delta_\tpX$ on $\dom(\delta_\tpX)$; note that different names of $x\in\tpX$ may give rise to different values of $y \in \tpY$.

The category of computable types with continuous functions is Cartesian closed, and the computable functions yield a Cartesian closed subcategory.
For any types $\tpX$, $\tpY$ there exist a canonical product type $\tpX\times\tpY$ with computable projections $\pi_{\tpX}:\tpX\times\tpY\fto\tpX$ and $\pi_{\tpY}:\tpX\times\tpY\fto\tpY$, and a canonical exponential type $\tpY^\tpX$ such that evaluation $\epsilon:\tpY^\tpX\times\tpX \rightarrow \tpY:(f,x)\mapsto f(x)$ is computable.
Since objects of the exponential type are continuous function from $\tpX$ to $\tpY$, we also denote $\tpY^\tpX$ by $\tpX\fto\tpY$ or $\tpCts(\tpX;\tpY)$; in particular, whenever we write $f:\tpX\fto\tpY$, we imply that $f$ is continuous.
There is a canonical equivalence between $(\tpX \times \tpY) \fto \tpZ$ and $\tpX \fto (\tpY \fto \tpZ)$ given by $\tilde{f}(x):\tpY\fto\tpZ:\tilde{f}(x)(y)=f(x,y)$.
%Further, there is a canonical equivalence between $\tpZ^{\tpX \times \tpY}$ and ${(\tpZ^\tpY)}^\tpX$ given by $f(x,y) = \tilde{f}(x)(y)$.

There are canonical types representing basic building blocks of mathematics, including the natural number type $\tpNat$ and the real number type $\tpRe$.
We use a three-valued logical type with elements $\{\false,\true,\indet\}$ representing \emph{false}, \emph{true}, and \emph{indeterminate} or \emph{unknowable}, and its subtypes the \emph{Boolean} type $\tpBl$ with elements $\{\false,\true\}$  and the \emph{Sierpinski} type $\tpSi$  with elements $\{\true,\indet\}$.
Given any type $\tpX$, we can identify the type $\tpOp(\tpX)$ of open subsets $U$ of $\tpX$ with $\tpX\fto\tpSi$ via the characteristic function $\chi_U$.
Further, standard operations on these types, such as arithmetic on real numbers, are computable.

A sequence $(x_n)$ is an \emph{effective Cauchy sequence} if $d(x_m,x_n) < \epsilon_{\max(m,n)}$ where $(\epsilon_n)_{n\in\N}$ is a known computable sequence with $\lim_{n\to\infty}\epsilon_n=0$, and a \emph{strong Cauchy sequency} if $\epsilon_n=2^{-n}$.
The limit of an effective Cauchy sequence of real number is computable.
%It is important to note that for \emph{any} uncountable type, equality $=$ is undecidable.

We shall also need the type $\tpHl \equiv \tpRe^{+,\infty}_<$ of positive real numbers with infinity under the lower topology.
The topology on the lower halfline $\tpHl$ is the toplogy of lower convergence, with open sets $(a,\infty]$ for $a\in\R^+$ and $\tpHl$ itself.
A representation of $\tpHl$ then encodes an increasing sequence of positive rationals with the desired limit. 
We note that the operators $+$ and $\times$ are computable on $\tpHl$, where we define $0\times\infty = \infty\times 0 = 0$, as is countable supremum $\sup:\tpHl^\omega \fto \tpHl$, $(x_0,x_1,x_2,\ldots)\mapsto\sup\{x_0,x_1,x_2,\ldots\}$.
Further, $\abs:\tpRe\fto\tpHl$ is computable, as is the embedding $\tpSi\hookrightarrow\tpHl$ taking $\true \mapsto 1$ and $\indet \mapsto 0$.
We let $\I_<$ be the unit interval $[0,1]$, again with the topology of lower convergence with open sets $(a,1]$ for $a\in[0,1)$ and $\I$ itself, and $\I_>$ the interval with the topology of upper convergence.

A \emph{computable metric space} is a pair $(\tpX,d)$ where $\tpX$ is a computable type, and $d:\tpX\times \tpX\fto \R^+$ is a computable metric, such that the extension of $d$ to $\tpX\times\clset(\tpX)$ defined by $d(x,A)=\inf\{d(x,y) \mid y\in A\}$ is computable as a function into $\R^{+,\infty}_<$.
This implies that given an open set $U$ we can compute $\epsilon>0$ such that $B_{\epsilon}(x) \subset U$, which captures the relationship between the metric and the open sets.
The effective metric spaces of~\cite{Weihrauch1999} are a concrete class of computable metric space.

A type $\tpX$ is \emph{effectively separable} if there is a computable function $\xi:\N\fto\tpX$ such that $\rng(\xi)$ is dense in $\tpX$.

Throughout this paper we shall use the term ``compute'' to indicate that a formula or procedure can be effectively carried out in the framework of type-two effectivity.
Other definitions and equations may not be possible to verify constructively, but hold from axiomatic considerations.

\section{Computable Measure Theory}

The main difficulty with classical measure theory is that Borel sets and Borel measures have very poor computability properties.
Although a computable theory of Borel sets was given in~\cite{Brattka2005}, the measure of a Borel set is in general not computable in $\R$.
However, we can consider an approach to measure theory in which we may only compute the measure of \emph{open} sets.
Since open sets are precisely those which can be approximated from inside, we expect to be able to compute lower bounds for the measure of an open set, but not upper bounds.
The above considerations suggest an approach which has become standard in computable measure theory, namely that using \emph{valuations}~\cite{JonesPlotkin1989,Edalat1995DTI,SchroderSimpson2006,Escardo2009}.
\begin{definition}[Valuation]
\label{defn:valuation}
The type of \emph{(continuous) valuations} on $\tpX$ is the subtype of continuous functions $\nu: \tpOp(\tpX) \fto \tpHl$ satisfying $\nu(\emptyset)=0$ and the \emph{modularity} condition $\nu(U)+\nu(V) = \nu(U\cup V) + \nu(U\cap V)$ for all $U,V \in \opset(\tpX)$.
\end{definition}
A valuation $\nu$ on $\tpX$ is \emph{finite} if $\nu(\tpX)$ is finite, \emph{effectively finite} if $\nu(\tpX)$ is a computable real number, and \emph{locally finite} if $\nu(U)<\infty$ for any $U$ with compact closure.
An effectively finite valuation computably induces an upper-valuation on closed sets $\bar{\nu}:\clset(\tpX)\fto \tpRe^{+}_{>}$ by $\bar{\nu}(A) = \nu(\tpX) - \nu(\tpX\setminus A)$.
The following proposition gives standard monotonicity and convergence properties for type of valuations.
\begin{proposition}
Let $\nu:\tpOp(\tpX)\fto\tpHl$ be continuous. Then $\nu$ satisfies the \emph{monotonicity} condition $\nu(U)\leq \nu(V)$ whenever $U\subset V$, and the \emph{continuity} condition $\nu\bigl(\bigcup_{n=0}^{\infty}U_n\bigr) = \lim_{n\to\infty} \nu(U_n)$ whenever $U_n$ is an increasing sequence of open sets.
\end{proposition}
\noindent
The proof is immediate from properties of continuous functions into $\tpHl$.
An immediate consequence is that $\nu(U)\leq \bar{\nu}(A)$ whenever $U\subset A$.

An explicit representation of valuations $\meas[0,1]$ on the unit interval was given in~\cite{Weihrauch1999} using the basic open sets $I_{a,b,r} = \{ \nu \in \meas[0,1] \mid \nu(a,b)>r \}$ for $a,b,r\in\Q$ with $0 \leq a<b \leq 1$ and $r>0$. Various representations for arbitrary spaces were given in~\cite{Schroder2007}.

The following theorem~\cite[Corollary~5.3]{Edalat1995DSM} shows that valuations and measures are equivalent on locally-compact Hausdorff spaces.
\begin{theorem}\label{thm:valuationborelmeasure}
%On a second countable locally compact Hausdorff space, finite measures and continuous valuations are in one to one correspondance.
On a countably-based locally-compact Hausdorff space, finite Borel measures and continuous valuations are in one-to-one correspondance.
\end{theorem}
In~\cite{AlvarezManilla2002}, it was shown that any continuous valuation on a locally compact sober space extends to a unique Borel measure.
This result provides a link with classical measure theory, but are not needed for a purely constructive approach; valuations themselves are the objects of study, and we only (directly) consider the measure of open and closed sets.

The following result shows that the measure of a sequence of small sets approaches zero.
Recall that a space $\tpX$ is \emph{regular} if for any point $x$ and open set $U$, there exists an open set $V$ and a closed set $A$ such that $x\in V\subset A\subset  U$.
\begin{lemma}\label{lem:decreasingopen}
Let $\tpX$ be a separable regular space, and $\nu$ a finite valuation on $\tpX$.
If $U_n$ is any sequence of open sets such that $U_{n+1}\subset U_n$ and $\bigcap_{n=0}^{\infty} U_n=\emptyset$, then $\nu(U_n)\to0$ as $n\to\infty$.
\end{lemma}
\begin{proof}
Since $\tpX$ is separable and regular, there exist open sets $V_{n,k}$ and closed sets $A_{n,k}$ such that $V_{n,k}\subset A_{n,k}\subset V_{n,k+1} \subset U_n$ and $\bigcup_{k\to\infty} V_{n,k}=U_n$.
Then $\lim_{k\to\infty} \nu(V_{n,k})=U_n$.
Suppose $\inf_{n\in\N}\nu(U_n)>\epsilon>0$.
Choose a sequence $\delta_n$ such that $\sum_{n=0}^{\infty}\delta_n = \delta<\epsilon$, and sets $V_n\subset A_n\subset U_n$ such that $\nu(V_n)\geq \nu(U_n)-\delta_n$.
Then $\bigcap_{n=0}^{N} A_n = U_N \setminus \bigcup_{n=0}^{N} (U_n\setminus A_n)$, so $\nu(\bigcap_{n=0}^{N} A_n) \geq \nu(U_N) - \sum_{n=0}^{N} \nu (U_n\setminus A_n) \geq \nu(U_N) - \sum_{n=0}^{N} (\nu (U_n) - \nu(A_n)) \geq \nu(U_N) - \sum_{n=0}^{N}\delta_n \geq \epsilon - \delta > 0$.
Since $\nu$ is upper-continuous on closed sets, $\nu(\bigcap_{n=0}^{\infty} A_n) \geq \epsilon - \delta > 0$.
Then $\bigcap_{n=0}^{\infty} A_n \neq \emptyset$, contradicting $\emptyset = \bigcap_{n=0}^{\infty} U_n \supset \bigcap_{n=0}^{\infty} A_n$.
\end{proof}

\begin{definition}
Given a sub-topology $\calV$ on $\tpX$ and a valuation $\nu$ on $\tpX$, a \emph{conditional valuation} is a function $\nu(\cdot|\cdot):\tpOp(\tpX)\times\calV\fto\tpHl$ such that $\nu(U\cap V)=\nu(U|V)\nu(V)$ for all $U\in\tpOp(\tpX)$ and $V\in\calV$.
\end{definition}
\noindent
Clearly, $\nu$ can be computed given $\nu$ restricted to $\calV$ and $\nu(\cdot|\cdot)$. 
The conditional valuation $\nu(\cdot|V)$ is uniquely defined if $\nu(V)\neq0$.
However, since $\nu(U\cap V):\R^+_<$ but $1/\nu(V):\R^{+,\infty}_>$, the conditional valuation $\nu(\cdot|V)$ cannot be \emph{computed} unless we are also given a set $A\in\tpCl(\tpX)$ such that $V\subset A$ and $\bar{\nu}(A\setminus V)=0$, in which case we have $\nu(U|V)=\nu(U\cap V)/\bar{\nu}(A)$.
We define the \emph{$\nu$-regular sets} as those for which $\nu(\partial V)=0$, so that $\nu(U|V)$ is continuous for open $U$ and non-null $\nu$-regular $V$.

Just as for classical probability, we say (open) sets $U_1,U_2$ are \emph{independent} if $\nu(U_1\cap U_2)=\nu(U_1)\nu(U_2)$.

We can define a notion of integration for positive lower-semicontinuous functions by the \emph{Choquet} or \emph{horizontal} integral; see~\cite{Tix1995,Lawson2004,Vickers2008}.
\begin{definition}[Lower horizontal integral]
Given a valuation $\nu:(\tpX\fto\tpSi)\fto\tpHl$, define the lower integral $(\tpX\fto\tpHl)\fto\tpHl$ by
\begin{equation} \textstyle \int_{\tpX}\!\psi\,d\nu = \sup \bigl\{ {\textstyle\sum_{m=1}^{n}} (p_{m}-p_{m-1}) \, \nu(\psi^{-1}(p_{m},\infty]) \mid (p_0,\ldots,p_n) \in \Q^* \text{ and } 0 = p_0 < p_1 < \cdots < p_n \bigr\} . \label{eq:lowerintegral} \end{equation}
\end{definition}
Note that we could use any dense set of computable positive real numbers, such as the dyadic rationals $\Q_2$, instead of the rationals in~\eqref{eq:lowerintegral}.
Since each sum is computable, and the supremum of countably many elements of $\tpHl$ is computable, we immediately obtain:
\begin{proposition}
Given names of a valuation $\nu$ in $(\tpX\fto\tpSi)\fto\tpHl$ and of a function $\psi$ in $\tpX\fto\tpHl$, the lower integral $\int_\tpX \psi\,d\nu$ is computable in $\tpHl$.
\end{proposition}
Note that although an alternative form for the sum is given through the equality
\[ {\textstyle\sum_{m=1}^{n}} (p_{m}-p_{m-1}) \, \nu(\phi^{-1}(p_{m},\infty]) = {\textstyle\sum_{m=1}^{n}} p_{m} \, \nu(\phi^{-1}(p_{m},p_{m+1}])  \]
where $p_{n+1}=\infty$,
the lower integral cannot be computed in this form since
\(  \nu(\phi^{-1}(p_{m},p_{m+1}]) = \nu(\phi^{-1}(p_{m},\infty]) - \nu(\phi^{-1}(p_{m+1},\infty]) \)
is uncomputable in $\H$.

It is fairly straightforward to show that the integral is linear,
\begin{equation} \textstyle \int_\tpX (a_1\psi_1+a_2\psi_2)\,d\nu= a_1\int_\tpX\psi_1\,d\nu+a_2\int_\tpX\psi_2\,d\nu  \end{equation}
for all $a_1,a_2\in \Hl$ and $\psi_1,\psi_2:\Hl\fto\Hl$.

If $\chif_U$ is the characteristic function of a set $U$, then
\( \textstyle \int_\tpX \chif_U \,d\nu= \nu(U) ,  \)
and it follows that if $\phi=\sum_{i=1}^{n} a_i\,\chif_{U_i}$ is a step function, then
%\[ \textstyle \raisebox{-0.3ex}{\Large$\int$}_{\!\!\!\tpX} \sum_{i=1}^{n} a_i\,\chif_{U_i}\,d\nu = \sum_{i=1}^{n} a_i\,\nu(U_i) .\]
\( \textstyle \int_\tpX\phi\,d\nu = \sum_{i=1}^{n} a_i\,\nu(U_i) .\)

Given a (lower-semi)continuous linear functional $\mu:(\tpX\fto\tpHl)\fto\tpHl$, we can define a function $\tpOp(\tpX)\fto\tpHl$ by
\( U \mapsto \mu(\chif_U) \) for $U\in\tpOp(\tpX)$.
By linearity, \[ \textstyle \mu(\chif_U)+\mu(\chif_V) = \mu(\chif_{U\cap V})+\mu(\chif_{U\cup V}) .  \]
Hence $\mu$ induces a valuation on $\tpX$.
We therefore obtain a computable equivalence between the type of valuations and the type of positive linear lower-semicontinuous functionals:
\begin{theorem}
\label{thm:lowermeasurevaluation}
The type of valuations $(\tpX\fto\tpSi)\fto\tpHl$ is computably equivalent to the type of continuous linear functionals $(\tpX\fto\tpHl)\fto\tpHl$.
\end{theorem}
%\begin{proof}
%Given a valuation $\nu:(\tpX\fto\tpSi)\fto\tpHl$, define linear $\mu:(\tpX\fto\tpHl)\fto\tpHl$ by $\mu(\psi)=\int_\tpX\psi(x)\,d\nu(x)$.
%Conversely, given $\mu$, define $\nu$ by $\nu(U) = \mu(\chif_U)$.
%\end{proof}
Types of the form $(\tpX\fto \tpT)\fto \tpT$ for a fixed type $\tpT$ form a \emph{monad}~\cite{Street1972} over $\tpX$, and are particularly easy to work with.

In~\cite[Section~4]{Edalat1995DTI}, a notion of integral $\ctsfn_\mathrm{bd}(\tpX;\R)\fto\R$ on continuous bounded functions was introduced based on the approximation by measures supported on finite sets of points.
Our lower integral on positive lower-semicontinuous functions can be extended to bounded functions as follows:
\begin{definition}[Bounded integration]
A valuation $\mu$ on $\tpX$ is \emph{effectively finite} if there is a (known) computable real $c\in\tpRe$ such that $\mu(\tpX)=c$.

An upper-semicontinuous function $f:\tpX\fto\tpRe_>$ is \emph{effectively bounded} if there is a (known) computable real $b\in\tpRe$ such that $f(x) < b$ for all $x\in\tpX$.

If $\mu$ is effectively finite and $f$ is effectively bounded, then the function $b-f:\tpX\fto\tpRe^+_<$ is computable (given names of $b$ and $f$), and we define the integral $\tpCts_{\mathrm{bd}}(\tpX;\tpRe_>) \fto \tpRe_>$ by
\[ \textstyle \int_\tpX f(x) \, d\mu(x) = b\,c - \int_\tpX \bigl(b-f(x)\bigr) \,d\mu(x) . \]
Similarly, if $f:\tpX\fto\tpRe_<$ has a computable lower bound $a$, we define the integral $\tpCts_{\mathrm{bd}}(\tpX;\tpRe_<) \fto \tpRe_<$ by
\[ \textstyle \int_\tpX f(x) \, d\mu(x) = \int_\tpX \bigl(a+f(x)\bigr) \,d\mu(x)  -  a\,c. \]

A continuous function $f:\tpX\fto\tpRe$ is \emph{effectively bounded} if there are a (known) computable reals $a,b\in\tpRe$ such that $a < f(x) < b$ for all $x\in\tpX$.
Then we define the integral $\tpCts_{\mathrm{bd}}(\tpX;\tpRe) \fto \tpRe$ by
\[ \textstyle \int_\tpX f(x) \, d\mu(x) = \int_\tpX \bigl(a+f(x)\bigr) \,d\mu(x)-a\,c = b\,c - \int_\tpX \bigl(b-f(x)\bigr) \,d\mu(x) . \]
\end{definition}
It is clear that the integrals defined above are computable in $\tpRe_{\gtrless}$ and that the lower and upper integrals agree if $f$ is continuous.
If $\tpX$ is compact, then any (semi)continuous function is effectively bounded, so the integrals always exist.

In order to define a valuation given a positive linear functional $\tpCts_{\mathrm{cpt}}(\tpX;\tpRe)\fto\tpRe$ on compactly-supported continuous functions, we need some way of approximating the characteristic function of an open set by continuous functions.
If $\tpX$ is \emph{effectively regular}, then given any open set $U$, we can construct an increasing sequence of closed sets $A_n$ such that $\bigcup_{n\to\infty} A_n=U$.
Further, a type $\tpX$ is \emph{effectively quasi-normal} if given disjoint closed sets $A_0$ and $A_1$, we can construct a continuous function $\phi:\tpX\fto[0,1]$ such that $\phi(A_0)=\{0\}$ and $\phi(A_1)=\{1\}$ using an effective Uryshon lemma; see~\cite{Schroder2009} for details.

We then have an effective version of the Riesz representation theorem:
\begin{theorem}\label{thm:effectiveriesz}
Suppose $\tpX$ is an effectively regular and quasi-normal type.
Then type of locally-finite valuations $(\tpX\fto\tpSi)\fto\tpHl$ is effectively equivalent to the type of positive linear functionals $\tpCts_{\mathrm{cpt}}(\tpX\fto\R)\fto\tpRe$ on continuous functions of compact support.
%Let $\tpX$ be an effectively locally-compact Hausdorff space. Then the types $L^+(C(\tpX;[0,1]);\R)$ and $C(\opset(\tpX);\R_<)$ are equivalent.
\end{theorem}
%The requirement that $\tpX$ be effectively quasi-normal means that not only there exists a sequence of \emph{continuous} functions converging to $\chif_U$ in $(\tpX\fto\tpHl)\fto\tpHl$, but also that such a sequence can be effectively constructed given a name of $U$.
%\begin{proof}
%\marginpar{Sketchy}
%A bounded positive linear functional $L:\tpCts_{\mathrm{bd}}(\tpX;\R)\fto\R$ effectively induces a valuation $\nu:\tpOp(\tpX)\fto\tpHl$ on $\tpX$ by \[ \nu(U) = %\sup\{ L(\phi) \mid \phi\in\tpCts(\tpX;[0,1]) \wedge \phi(\tpX\setminus U)=\{0\} \} . \]
%This construction extends to locally-finite measures and compactly-supported functions $C_{\mathrm{cpt}}(\tpX\fto\R)$.
%\end{proof}

We consider lower-semicontinuous functionals $(\tpX \fto \tpHl) \fto \tpHl$ to be more appropriate as a foundation for computable measure theory than the continuous functionals $(\tpX \fto \tpRe) \fto \tpRe$, since the equivalence given by Theorem~\ref{thm:lowermeasurevaluation} is entirely independent of any assumptions on the type $\tpX$ whereas the equivalence of Theorem~\ref{thm:effectiveriesz} requires extra properties of $\tpX$ and places restrictions on the function space.

A similar monadic approach to probability measures~\cite{Escardo2009} based on type theory identified the type of probability measures on the Cantor space $\Omega=\{0,1\}^\omega$ with the type of integrals $(\Omega \fto \I) \fto \I$ where $\I=[0,1]$ is the unit interval.
%However, we prefer to work with $\H$ and the topology of lower convergence due to the equivalence with valuations discussed above.

\section{Computable Random Variables}

A computable theory of random variables should, at a minimum, enable us to perform certain basic operations, including:
\begin{enumerate}[(i)]
 \item Given a random variable $X$ and open set $U$, compute lower-approximation to $\Pr(X\in U)$. \label{property:distribution}
 \item Given random variables $X_1,X_2$, compute the random variable $X_1\times X_2$ giving the joint distribution. \label{property:product}
 \item Given a random variable $X$ and a continuous function $f$, compute the image $f(X)$. \label{property:image}
 \item Given a sequence of random variables $X_1,X_2,\ldots$ converging effectively in probability, compute a limit random variable $X_\infty = \lim_{m\to\infty} X_m$. \label{property:limit}
 \item Given a probability distribution $\nu$ on a sufficiently nice space $\tpX$, compute a random variable $X$ with distribution $\nu$. \label{property:realisation}
 \item Given random variables $X_1,X_2$, compute a random variable $X_1\otimes X_2$ such that $\Pr(X_1\otimes X_2)\in(U_1\times U_2)=\Pr(X_1\in U_2) \Pr(X_2\in U_2)$. \label{property:independence}
\end{enumerate}
Property~(\ref{property:distribution}) states that we can compute the distribution of a random variable, while property~(\ref{property:product}) implies that a random variable is \emph{more} than its distribution; it also allows us to compute its joint distribution with another random variable.
Property~(\ref{property:image}) also implies that for random variables $X_1,X_2$ on a computable metric space $(\tpX,d)$, the random variable $d(X_1,X_2)$ is computable in $\R^+$, so the probability $\Pr(d(X_1,X_2)<\epsilon)$ is computable in $\tpIv_<$, and $\Pr(d(X_1,X_2)\leq\epsilon)$ is computable in $\tpIv_>$.
Property~(\ref{property:limit}) is a completeness property and allows random variables to be approximated.
Property~(\ref{property:realisation}) shows that random variables can realise a given distribution, while property~(\ref{property:independence}) shows that independent random variables can be constructed realising a given distribution.
These properties are similar to those used in~\cite{Kersting2008}.

%Since $P:(\Omega\fto\Hl)\fto \Hl$ induces a valuation $P:\opset(\Omega)\fto \Hl$, $X$ should define a computable function $\opset(\tpX)\rightarrow \opset(\Omega)$.
%If $X$ is a single-valued function $\Omega\rightarrow \tpX$, this is precisely the condition that $X$ is continuous.
The standard approach to probability theory used in classical analysis is to define random variables as measurable functions over a base probability space.
Given types $\tpX$ and $\tpY$, a representation of the Borel measurable functions $f:\tpX\fto\tpY$ was given in~\cite{Brattka2005}, but this does not allow one to compute lower bounds for the measure of $f^{-1}(V)$ for $V\in\tpOp(\tpY)$.
Ideally, one would like a representation of bounded measurable functions $f:\tpX\fto \R$ such that for every finite measure $\mu$ on $\tpX$, the integral $\int_\tpX f(x) \, d\mu(x)$ is computable.
But then $f(y) = \int_\tpX f(x) \, d\delta_y(x)$ would be computable, so $f$ would be continuous.
Any effective approach to measurable functions and integration must therefore take some information about the measure into account.

In the approach of~\cite{BishopCheng1972}, a notion of full-measure set was given independently of a specific measure, but this introduces additional technical details.
In~\cite{BishopBridges1985}, integrable functions are defined as limits in an \emph{integration space} of functions, and measurable functions through approximation by integrable function.
In the approach of \cite{GoubaultLarrecqVaracca2011} a notion of continuous random variable was introduced as a continuous function on $\supp(\nu)$, where $\nu$ is a valuation on the Cantor space $\{0,1\}^\omega$. However, in order to define a joint distribution, we need to fix the measure $\nu$, but for fixed $\nu$, the set of continuous functions is not expressive enough. For example, using the standard probability measure $P$ on $\{0,1\}^\omega$, there is no continuous total function $X:\{0,1\}^\omega \fto \{0,1\}$ such that $\Pr(X(\omega)=1)=1/3$.
In~\cite{Spitters2003}, a type of integrable real-valued functions is defined as the completion of the continuous functions under the metric defined by $d(f,g)=\int_\tpX |f(x)-g(x)|\,d\mu(x)$, and extended to a type of measurable functions.
This approach is natural, constructive, and allows for integrals of measurable functions to be computed; it is this approach we shall use here.

We will consider random variables on a fixed probability space $(\Omega,P)$.
Since any probability distribution on a Polish space is equivalent to a distribution on the standard Lesbesgue-Rokhlin probability space~\cite{Rokhlin1952}, it is reasonable to take the base space to be the Cantor space $\Sigma=\{0,1\}^\omega$ and $P$ the standard measure.

\subsection{Measurability}

\begin{definition}[Continuous random variable]
An \emph{continuous random variable} on $(\Omega,P)$ with values in $\tpX$ is a continuous function $X:\Omega\fto \tpX$ .
\end{definition}
We will sometimes write $\P(X\in U)$ as a shorthand for $P(\{\omega\in\Omega \mid X(\omega)\in U\})$.
Continuous random variables $X$ and $Y$ are considered equal if $P(\{\omega \in \Omega \mid X(\omega) \neq Y(\omega)\})=0$. In other words, $X$ and $Y$ are \emph{almost-surely equal}.

Suppose $\tpX$ is a Polish space, i.e. a space which is separable and complete under the metric $d$.
Define  the \emph{Fan metric} on continuous random variables by
\begin{equation} \label{eq:randomvariablefanmetric}
 \begin{aligned} d(X,Y) &= \sup\!\big\{ \varepsilon\in\Q^+ \mid \ P\big(\{\omega\in\Omega \mid d(X(\omega),Y(\omega))>\varepsilon\}\big) > \varepsilon\big\} \\
   &= \inf\!\big\{ \varepsilon \in \Q^+ \mid \ P\big(\{\omega\in\Omega \mid d(X(\omega),Y(\omega))\geq\varepsilon\}\big) < \varepsilon\big\} . \end{aligned}
\end{equation}
Given the probability distribution (valuation) $P$ and a computable metric $d:\tpX\times \tpX\fto\R^+$, the Fan metric on continuous random variables is easily seen to be computable.
The convergence relation defined by the Fan metric corresponds to convergence in probability.
As an alternative to using the Fan metric, we can consider a uniform structure on $\tpX$, or, if the metric $d$ on $\tpX$ is bounded, the distance \( \textstyle d(X,Y) := \int_\Omega d(X(\omega),Y(\omega)) \, dP(\omega) . \)

\begin{definition}[Measurable random variable]
The type of \emph{measurable random variables} is the effective completion of the type of continuous random variables under the Fan metric~\eqref{eq:randomvariablefanmetric}.
We write $X:\Omega \mfto \tpX$ if $X$ is a measurable random variable taking values in $\tpX$, and let $\rand(\tpX)$ be the type of measurable random variables with values in $\tpX$.
\end{definition}
\noindent
In other words, a random variable is represented  by a sequence $(X_0,X_1,X_2,\ldots)$ of continuous random variables satisfying $d(X_m,X_n)  < 2^{-\min(m,n)}$, and two such sequences are equivalent (represent the same random variable) if $d(X_{1,n},X_{2,n}) \to 0$ as $n\to\infty$.

By standard results on the completion, the Fan metric on continuous random variables extends computably to measurable random variables.
For if $m>n$, then $|d(X_m,Y_m) - d(X_n,Y_n)| \leq d(X_m,X_n)+d(Y_m,Y_n) \leq 2 \cdot 2^{-n}$, so $(d(X_m,Y_m))_{m\in\N}$ is an effective Cauchy sequence converging to a value we define as $d(\lim_{n\to\infty}X_n,\lim_{n\to\infty}Y_n)$.
Further, if $(X_0,X_1,X_2,\ldots)$ is an effective Cauchy sequence converging to $X_\infty$, then $d(X_n,X_\infty)\leq 2^{-n}$.

\begin{remark}
Although a measurable random variable $X$ is \emph{defined} relative to the underlying space $\Omega$, we cannot in general actually \emph{compute} $X(\omega)$ in any meaningful sense for fixed $\omega\in\Omega$! The expression $X(\omega)$ only makes sense for random variables \emph{given} as continuous functions $\Omega\fto \tpX$.
\end{remark}

%\begin{definition}
%A Cauchy sequence $(x_n)_{n\in\N}$ converges \emph{strongly} if $d(x_m,x_n) < 2^{-\min(m,n)}$ for all $m,n$.
%A Cauchy sequence $(x_n)_{n\in\N}$ converges \emph{effectively} there is a computable sequence $\epsilon_n$ such that it is possible to prove $d(x_m,x_n) < \epsilon_{\min(m,n)}$ for all $m,n$.
%A sequence $x_n\to x_\infty$ converges \emph{effectively} if $d(x_n,x_\infty)$ is computable in $\R^+_>$ and $\lim_{n\to\infty}d(x_n,x_\infty)=0$.
%\end{definition}

It will sometimes be useful to consider random variables taking only finitely many values.
\begin{definition}[Simple random variable]
An \emph{simple random variable} on $(\Omega,P)$ with values in $\tpX$ is a continuous function $X:\Omega\fto \tpX$ which takes finitely many values.
\end{definition}
\noindent 
Clearly, if the base space $\Omega$ is connected, then any simple random variable is constant, but for base space $\Sigma$, any continuous random variable can be effectively approximated by simple random variables, which immediately yields effective approximation by measurable random variables.
\begin{lemma}
Given any continuous function $X:\Sigma\fto\tpX$, we can compute a sequence of simple functions $X_m$ converging effectively to $X$ in the uniform metric.
\end{lemma}
\begin{proof}
Take $I_m(\omega)=\omega|_m 0^\omega$ for all $\omega\in\Sigma$, $m\in\N$, and let $X_m = X \circ I_m$.
Then $d(X_m,X)=\sup_{\omega\in\Sigma} d(X(I_m(\omega)),X(\omega))$ is computable, being the supremum of a continuous function over a compact set.
Further, $d(X_m,X)\to 0$ as $m\to\infty$ since $X$ is uniformly continuous, so $X_m$ converges effectively.
\end{proof}

It is also useful to consider more general classes of random variables by allowing for partial functions on a full-measure set. 
This is important if the base space $\Omega$ is connected, but $\tpX$ is path-connected.
\begin{definition}[Piecewise-continuous random variable]\label{defn:piecewisecontinuousrandomvariable}
A \emph{piecewise-continuous random variable} on $(\Omega,P)$ with values in $\tpX$ is a continuous partial function $X:\Omega\pfto \tpX$ such that ($\dom(X)\in\tpOp(\Omega)$ and $P(\dom(X))=1$.
\end{definition}
\noindent
We use the terminology ``piecewise-continuous'' since $X:\omega\pfto \tpX$ may arise as the restriction of a piecewise-continuous function to its continuity set.
%Note that the Fan metric extends naturally to piecewise-continuous random variables.
%\begin{proposition}
%Any piecewise-continuous random variable defined on a full-measure open set is a measurable random variable.
%\end{proposition}
%\noindent
%This result is a special case of the more general Proposition~\ref{prop:almostsurelycontinuousrandomvariable}.
%\begin{proof}
%Let $X$ be a piecewise-continuous random variable.
%Let $x_*$ be an arbitrary point of $\tpX$; a suitable point can be effectively constructed as $x_*=X(\omega_*)$ for some $\omega_*\in\dom(X)$.
%Compute an increasing sequence of clopen sets $W_n\subset\dom(X)$ such that $\mu(W_n)>1-2^{-n}$.
%Define $X_n(\omega)=X(\omega)$ for $\omega\in W_n$, and $X_n(\omega)=x_*$ otherwise.
%Then clearly $P(d(X_m,X_n)>0) \leq \mu(\Omega\setminus W_n) \leq 2^{-n}$ for $m\leq n$, so $(X_n)_{n\in\N}$ is an effective Cauchy sequence of random variables, and converges in the Fan metric to $X$.
%\end{proof}

By~\cite[Theorem~2.2.4]{Weihrauch1999}, machine-computable functions $\{0,1\}^\omega\fto\{0,1\}^\omega$ are defined on a $G_\delta$-subsets of $\{0,1\}^\omega$.
Indeed, any function into a metric space is continuous on a $G_\delta$ set of points.
This makes functions defined and continuous on a full-measure $G_\delta$-subset of $\Omega$ a natural class of random variables.
\begin{definition}[Almost-surely continuous random variable]\label{defn:almostsurelycontinuousrandomvariable}
An \emph{almost-surely-continuous random variable} on $(\Omega,P)$ with values in $\tpX$ is a continuous partial function $X:\Omega\pfto \tpX$ such that $\dom(X)$ is a $G_\delta$ set and $P(\dom(X))=1$, where $P(\bigcap_{n\in\N} U_n)=1$ if $P(U_n)=1$ for all $n$.
\end{definition}
The following result shows that almost-surely continuous random variables are measurable random variables for the base space $\Sigma$.
\noindent
\begin{proposition}\label{prop:almostsurelycontinuousrandomvariable}
Suppose the base probability space is $\{0,1\}^\omega$.
Then any almost-surely-continuous random variable defined on a full-measure open set is effectively a measurable random variable.
However, not all measurable random variables are almost-surely continuous.
\end{proposition}
\begin{proof}
Given a almost-surely-continuous random variable $X$, we can construct a sequence of continuous random variables converging weakly to $X$.
Let $\xi:\{0,1\}^\omega\fto\{0,1\}^\omega$ be the code of machine-computable function representing $X$, so $X=\delta\circ\xi$ where $\delta:\{0,1\}^\omega\psfto \tpX$ is a representation of $\tpX$, and $\dom(X)=\dom(\xi)$.
Fix $\epsilon=2^{-n}>0$.
Consider the set of $\omega$ on which $\xi$ is defined and for some $\delta$ maps the $\delta$-ball abound $\omega$ into the $\epsilon$-ball about $\xi(\omega)$.
Since $\xi$ is continuous on its domain, this set is $\dom(\xi)$.
Hence on some full-measure open set, $\xi$ is provably defined up to error $2^{-n}$.
Since this set is a countable union of cylinder sets, we can compute $X_n$ agreeing with $X_\infty$ up to $2^{-n}$ on a set of measure $1-2^{-n}$.

Conversely, we can define a strong Cauchy sequence $X_n$ of piecewise-continuous random variable taking values in $\{0,1\}$ such that $X_n=1$ on a decreasing sequence of closed sets $W_n$ of measure $(1+2^{-n})/2$ whose limit is a Cantor set.
Then $\lim_{n\to\infty}X_n$ is discontinuous on a set of positive measure.
\end{proof}

Similarly to the notion of measurable function, we can define the notion of measurable set.
\begin{definition}[Measurable set]
A \emph{measurable set} $A$ in $\Omega$ (more precisely, the characteristic function $\chi_A$ of a measurable set $A$) is a measurable random variable in $\tpIv=[0,1]$ such that $\Pr(\chi_A\in\{0,1\})=1$.
\end{definition}
\noindent
If $\Omega=\Sigma$, then the characteristic function $\chi_A$ is the limit in probability of $\chi_{A_n}$ for clopen sets $A_n$.
Equivalently, any measurable set is a limit of an effective Cauchy sequence of clopen sets $(A_n)_{n\in N}$ under the metric $d(A_m,A_n)=P(A_m \triangle A_n)$.

In classical measure theory, it is also useful to consider the indicator function of a set of values of a random variable, defined as
\begin{equation} I[X\in S]:\Omega\mfto\{0,1\}:\omega \mapsto \begin{cases}1\text{ if } X(\omega)\in S; \\ 0 \text{ if } X(\omega)\not\in S . \end{cases} \end{equation}
If $X$ is a continuous random variable, and $U$ is open, then $I[X\in U]$ is computable as a function $\Omega\fto [0,1]_<$, and if $A$ is closed, then $I[X\in A]$ is computable in $\Omega\fto[0,1]_>$.
These indicator functions cannot be seen as random variables as the range spaces are not Hausdorff.
Indicator functions for measurable random variables taking values in the Polish space $\{0,1\}\subset\R$ are only computable as measurable random variables for clopen sets as the following example shows:
\begin{example}[Uncomputability of indicator functions]
Let $X$ be a random variable and $U$ an open set.
Suppose $I[X\in U]$ were to be computable as a measurable random variable given $X$ and $A$.
Then $\Pr(X\in A) = \Pr(I[X\in U]\geq \tfrac{1}{2})$ would be computable in $[0,1]_>$, so $\Pr(X\in U)$ would be computable in $[0,1]$.
Taking $X=\delta_x$ gives $\Pr(X\in U) = 1$ if $x\in U$ and $0$ if $x\not\in U$, so $U$ would be effectively clsed.
\end{example}
\noindent
However, we shall see that for an open set $U$, the indicator function $I[X\in U]$ induces a valuation on $\Omega$ which is computable.
This makes indicator functions useful when we are only interested in information about probabilities and expectations, such as the submartingale inequality~\eqref{eq:discretesubmartingaleinequality}.

\subsection{Distribution}

We now consider the probability distribution of a measurable random variable.
Let $\tpX$ be a computable metric space.
For a closed set $A$, define $\clN_\epsilon(A):=\{x\in \tpX \mid d(x,A)\leq \varepsilon\}$, and for an open set $U$ define $I_\varepsilon(U):=\tpX\setminus(\clN_\varepsilon(\tpX\setminus U))=\{x\in U \mid \exists \delta>0, B(x,\varepsilon+\delta)\subset U\}$.
Since $d(x,A)$ is computable in $R^{+,\infty}_<$ by definition of a computable metric space, $\clN_\epsilon(A)$ is computable as a closed set, so $I_\varepsilon(U)$ is computable as an open set.
Note that $I_{\varepsilon_1+\varepsilon_2}(U) \subset I_{\varepsilon_1}(I_{\varepsilon_2}(U))$.
\begin{definition}[Distribution of a measurable random variable]
For a measurable random variable $X$, define its \emph{distribution} by
%\[ \Pr(X\in U) = \sup\{ P(\{\omega\in\Omega \mid Y(\omega)\in V\}) - \varepsilon \mid \varepsilon\in\Q^+,\ V\subset I_\varepsilon(U),\ d(Y,X)<\varepsilon\} . \]
\[ \Pr(X\in U) = \sup\{ P(Y\in V) - \varepsilon \mid \varepsilon\in\Q^+,\ V\subset I_\varepsilon(U),\ d(Y,X)<\varepsilon\} . \]
where $Y$ ranges over continuous random variables and $V$ over open sets.
\end{definition}

\begin{theorem}[Computability of distribution]
\label{thm:randomvariabledistribution}
Suppose $(\tpX,d)$ is a computable metric space.
The distribution of a measurable random variable $X$ taking values in $\tpX$ is a valuation, and is computable from a name of $X$.
If $X$ is a continuous random variable, then $\Pr(X\in U) = P(\{\omega\in\Omega \mid X(\omega)\in U\})$.
\end{theorem}
\begin{proof}
Suppose $X,Y$ are continuous random variables, $d(X,Y)<\epsilon$ and $V\subset I_\epsilon(U)$.
Then $\P(X\in U) \geq \P(Y\in V \wedge d(X,Y)<\epsilon) \geq \P(Y\in V) - \P(d(X(\omega),Y(\omega))\geq\epsilon) \geq \P(Y(\omega)\in V)-\epsilon$.

Now take $(X_n)_{n\in\N}$ to be any sequence of continuous random variables converging effectively to a measurable random variable $X$.
By definition of $\Pr(X\in U)$, we have $\Pr(X\in U) \geq \P(X_m\in I_{2^{-m}}) - 2^{-m}$ for all $m$

Fix $\delta>0$ and take a continuous random variable $Y$ such that $d(X,Y)<\epsilon$ and $\P(Y\in I_\epsilon(U)) - \epsilon > \Pr(X \in U) - \delta$.
By taking $m$ sufficiently large, we can ensure that $\P(Y\in I_{2^{1-m}+\epsilon}(U)) - \epsilon > \Pr(X\in U) - \delta$.
Then since $d(X_m,Y)\leq \epsilon+2^{-m}$, we have $\P(X_m\in I_{2^{-m}})- 2^{-m} \geq \P(Y\in I_{\epsilon+2^{1-m}}(U)) - (\epsilon+2^{1-m}) > \Pr(X\in U) - (\delta +2^{1-m})$.
Since $\delta$ is arbitrary and $m$ may be taken arbitrarily large, we have $\Pr(X\in U) = \sup_{n\in\N} \P(X_n\in I_{2^{-n}}(U))-2^{-n}$, so is computable in $\tpIv_<$.

If $X$ is a continuous random variable, $\P(X\in I_\epsilon(U))\nearrow \P(X\in U)$ as $\epsilon\to 0$ by continuity, so taking $X_n=X_\infty=X$ above, we have $\Pr(X\in U)=\lim_{n\to\infty} \P(X\in I_\epsilon(U))-\epsilon = \P(X\in U)$.
\end{proof}
\begin{remark}
Although the notation $\Pr(X\in U)$ suggests that we can define a measurable random variable $X$ as a \emph{function} from $\Omega$ to $\tpX$, such a function could not be constructed in general, and is not required to compute probabilities.
\end{remark}
\begin{corollary}
If $X$ is a random variable and $A$ is a closed set, then $\Pr(X\in A)$ is computable in $[0,1]_>$.
%In particular, statements of the form $\Pr(X\in A) \geq \epsilon$ can be falsified.
\end{corollary}

We now show that we can \emph{construct} a random variable with a given \emph{distribution}.
The result below is an variant of~\cite[Theorem~1.1.1]{HoyrupRojas2009IC}, which shows that any distribution is effectively measurably isomorphic to a distribution on $\{0,1\}^{\omega}$, and the proof is similar.

We first prove the following generally-useful decomposition result, which is essentially a special case of the effective Baire category theorem~\cite{YasugiMoriTsujii1999,Brattka2001}.
\begin{lemma}
Let $X$ be an effectively separable computable metric space, and $\mu$ be a measure on $X$.
Then given any $\epsilon>0$, we can compute a topological partition $\mathcal{B}$ of $X$ such that $\diam(B)<\epsilon$ for all $B\in\mathcal{B}$, and $\mu(X\setminus\bigcup\mathcal{B})=0$.
\end{lemma}
\begin{proof}
For any $\delta>0$, and any $x\in\tpX$, $\{r>0 \mid \mu(\clB(x,r)\setminus B(x,r))<\delta\}$ is a computable open dense set.
We can therefore construct a sequence of rationals $q_k$ such that $|q_k-q_{k+1}|<2^{-k-1}$ and $\mu(\clB(x,q_k)\setminus B(x,q_k))<2^{-k}$.
Then taking $r_\delta(x)=\lim_{k\to\infty} q_k$ yields a suitable radius.

Since $X$ is effectively separable, it has a computable dense sequence $(x_n)_{n\in\N}$.
For $\epsilon>0$, we take as topological partition the sets $B(x_n,r_\epsilon(x_n)) \setminus \bigcup_{m=0}^{n-1} \clB(x_m,r_\epsilon(x_m))$ for $n\in\N$.
\end{proof}
\begin{theorem}
Let $\tpX$ be a computable metric space, and $\nu$ be a valuation on $\tpX$.
Then we can compute a measurable random variable $X$ on base space $\{0,1\}^\omega$ such that for any open $U$, $\Pr(X\in U)=\nu(U)$.
\end{theorem}
\begin{proof}
For each $n$, construct a countable topological partition $\mathcal{B}_n$ such that each $B\in B_n$ has radius at most $2^{-n}$, and $\sum_{B\in\mathcal{B}_n}\mu(\partial B)=0$ such that $\nu(\bigcup\mathcal{B}_n\} > 1-2^{-n-1}$.
By taking intersections if necessary, we can assume that each $\mathcal{B}_{n+1}$ is a refinement of $\mathcal{B}_n$.

We now construct random variables $X_n$ as follows.
Suppose we have constructed cylinder sets $W_{n,m}\subset\{0,1\}^\infty$ such that $P(W_{n,m})<\mu(B_{n,m})$ and $\sum_{m} P(W_{n,m})>1-2^{-n}$.
Since $B_{n,m}$ is a union of open sets $\{B_{n+1,m,1},\ldots,B_{n+1,m,k}\}\subset\mathcal{B}_{n+1}$, we can effectively compute dyadic numbers $p_{n+1,m}$ such that $p_{n+1,m,k}<\mu(B_{n+1,m,k})$ and $\sum_k p_{n+1,m,k}\geq P(W_{n,m})$.
We then partition $W_{n,m}$ into cylinder sets $W_{n+1,m,k}$ each of measure $p_{n,m,k}$.
We take $X_{n+1}$ to map $W_{n+1,m,k}$ to a point $x_{n+1,m,k}\in B_{n+1,m,k}$.
It is clear that $X_{n}$ is a strongly-convergent Cauchy sequence, so converges to a random variable $X_\infty$.

It remains to show that $\Pr(X_\infty \in U) = \mu(U)$ for all $U\in\opset(X)$.
This follows since for given $n$ we have $\Pr(X_n \in U) > \mu(I_{2^{1-n}}(U))-2^{-n} \nearrow \mu(U)$ as $n\to\infty$.
\end{proof}

%\begin{conjecture}
%Let $X$ be a random variable, and $U_n$ be a decreasing sequence of open sets such that $\bigcap_{n=1}^{\infty} U_n = \emptyset$.
%Then $\Pr(X\in U_n) \to 0$ as $n\to\infty$.
%\end{conjecture}

%If $X,Y$ are piecewise-continuous random variables on the same probability space $(\Omega,P)$, and $\star$ is a computable operator, then $X \star Y$ is an piecewise-continuous random variable on $(\Omega,P)$, since $\dom(X \star Y) = \dom(X) \cap \dom(Y)$ and $\Pr(\dom(X) \cap \dom(Y)) = \Pr(\dom(X)) + \Pr(\dom(Y)) - \Pr(\dom(X) \cup \dom(Y)) = 1$.
%Indeed, since the countable intersection of full-measure subsets also has full measure, we can consider operators on countable sets of random variables.

\subsection{Product and Image}

If $X_1$, $X_2$ are continuous random variables, define the product $X_1\times X_2$ as the functional product
\begin{equation*} (X_1\times X_2)(\omega) = (X_1(\omega),X_2(\omega)) .\end{equation*}
Our second main result is that we can also compute products of measurable random variables.
\begin{theorem}[Computability of products]
\label{thm:productrandomvariable}
If $X_n$ and $Y_n$ are effective Cauchy sequences of continuous random variables, then $X_n \times Y_n$ is an effective Cauchy sequence.
\end{theorem}
\begin{proof}
If $m>n$, then $\P(d(X_{m+1} \!\times\! Y_{m+1},X_{n+1} \!\times\! Y_{n+1}) \geq 2^{-n}) = \P(d(X_{m+1},X_{n+1}) \geq 2^{-n} \vee d(Y_{m+1},Y_{n+1})\geq 2^{-n}) \leq \P(d(X_{m+1},X_{n+1})\geq 2^{-(n+1)})+\P(d(Y_{m+1},Y_{n+1})\geq 2^{-(n+1)}) \leq 2^{-(n+1)} + 2^{-(n+1)} = 2\cdot 2^{-(n+1)} = 2^{-n}$.
Hence $d(X_m\times Y_m,X_n\times Y_n) \leq 2^{-n}$ as required.
\end{proof}

If $X:\Omega \pfto \tpX$ is a continuous random variable, and $f:\tpX\fto \tpY$ is continuous, then $f\circ X$ is a continuous random variable.
By taking limits of effective Cauchy sequences, it is clear that we can define the image of a measurable random variable under a \emph{uniformly} continuous function.
However, it is possible to compute the image under an arbitrary continuous function.
\begin{theorem}[Continuous mapping]
\label{thm:imagerandomvariable}
If $X_n$ is an effective Cauchy sequence of continuous random variables taking values in $\tpX$ and $f:\tpX\fto\tpY$ is continuous, then $f \circ X_n$ is an effective Cauchy sequence.
Further, % if $X_\infty:=\lim_{n\to\infty}X_n$ is continuous, then so is $\lim_{n\to\infty} f(X_n)$, and
for any open $V\subset \tpY$,
$\Pr(\lim_{n\to\infty} f\circ X_n \in V) = \Pr(\lim_{n\to\infty} X_n \in f^{-1}(V))$.
\end{theorem}
\noindent
The proof is based on the non-effective version of this result from~\cite{MannWald1943}.
%See also Wikipedia "Continuous mapping theorem"
\begin{proof}
Consider the open subsets of $\tpX$ defined as
\[ B_{\delta,\epsilon}(f) = \big\{x\in \tpX\ \mid\exists y\in \tpX,\ d(x,y)<\delta \wedge d(f(x),f(y))>\epsilon\big\}. \]
Since $f$ is continuous, for all $\epsilon>0$, $\bigcap_{\delta>0} B_{\delta,\epsilon}(f)=\emptyset$.

For all $x,y,z\in \tpX$,
\[ \begin{aligned}  d(f(x),f(y))>\epsilon &\implies d(f(x),f(z))>\epsilon/2 \,\vee\, d(f(y),f(z))>\epsilon/2 \\ & \implies d(x,z)\geq \delta \vee d(y,z)\geq \delta \vee z\in B_{\delta,\epsilon/2}(f)  . \end{aligned}\]
Hence for random variables $X,Y,Z$,
\[\Pr\big(d(f(X),f(Y))>\epsilon\big) \leq \Pr\big(d(X,Z)>\delta\big) + \Pr\big(d(Y,Z)>\delta\big) + \Pr(Z\in B_{\delta,\epsilon/2}(f)). \]

Let $X_n$ be an effective Cauchy sequence of continuous random variables with limit $X$.
Fix $\epsilon>0$.
By continuity of the discribution of $X$, we have $\Pr(X\in B_{\delta,\epsilon/2})<\epsilon/2$ for sufficiently small $\delta$, and by computability of the distribution, we can effectively find such a $\delta$.
Take $N(\epsilon)$ so that $2^{-N(\epsilon)}<\min\{\delta,\epsilon/4\}$.
Then for $n\geq N(\epsilon)$, we have $\Pr\big(d(X,X_n)>\delta)<\Pr\big(d(X,X_n)>2^{-n})<2^{-n}<\epsilon/4$.
Therefore if $m,n\geq N(\epsilon)$, we have $\Pr\big(d(f(X_m),f(X_n))>\epsilon\big) \leq \Pr\big(d(X_m,X)>\delta\big) + \Pr\big(d(X_n,X)>\delta\big) + \Pr(X\in B_{\delta,\epsilon/2}(f)) < \epsilon$.
Thus $f(X_n)$ is an effective Cauchy sequence, satisfying $d(f(X_n),f(X_m))<\epsilon$ whenever $m,n\geq N(\epsilon)$.

Fix $\epsilon>0$, and choose $\delta$ such that $\Pr(f(X)\in V) < \Pr(f(X)\in I_{\delta}(V)) + \epsilon/2$.
For $n$ sufficiently large, we have $d(f(X_n),f(X))<\min\{\delta,\epsilon/2\}$.
Then $\Pr(f(X)\in V) \geq \P(f(X_n)\in I_\epsilon(V))-\epsilon = \P(X_n\in f^{-1}(I_\epsilon(V)))-\epsilon$.
Taking $n\to\infty$ gives $\Pr(f(X)\in V) \geq \Pr(X\in f^{-1}(I_\epsilon(V)))-\epsilon$, and taking $\epsilon\to0$ given $f^{-1}(I_\epsilon(V))\to f^{-1}(V)$, so $\Pr(f(X)\in V) \geq \Pr(X\in f^{-1}(V))$.
For the reverse inequality, for $n$ sufficiently large, we have $d(f(X_n),f(X))<\min\{\delta,\epsilon/2\}$, so $\Pr(f(X_n)\in V) \geq \Pr(f(X)\in I_{\delta}(V)) - \epsilon/2$.
Then $\Pr(f(X)\in V) < \Pr(f(X_n)\in V) + \epsilon = \P(X_n \in f^{-1}(V))+ \epsilon$.
Taking $n\to\infty$ gives $\Pr(f(X)\in V) \leq \Pr(X \in f^{-1}(V)) + \epsilon$, and since $\epsilon$ is arbitrary, $\Pr(f(X)\in V) \leq \Pr(X \in f^{-1}(V))$.
\end{proof}

If $X$ is a measurable $\tpX$-valued random variable and $f:\tpX\fto \tpY$ where $\tpY$ is not a metric space, then $f(X)$ is not defined as a random variable, since we have only a notion of random variables on metric spaces.
However, the distribution of $f(X)$ \emph{is} well-defined and computable, with $\Pr(f(X)\in V) := \Pr\bigl(X\in f^{-1}(V)\bigr)$ for $V\in\opset(V)$.
Indeed, we can compute the joint distribution of $f(X)$ and $Z\in\tpRV(\tpZ)$ by
$\Pr(f(X)\in V \wedge Z \in W) := \Pr\bigl(X\in f^{-1}(V) \wedge Z\in W\bigr)$ for $V\in\opset(\tpY)$ and $W\in\opset(\tpZ)$.

%\subsection{Construction}

\subsection{Expectation}

The expectation of a random variable is not continuous in the weak topology; for example, we can define continuous random variables $X_n$ taking value $2^n$ on a subset of $\Omega$ of measure $2^{-n}$, so that $X_n\to0$ but $\Ex(X_n)=1$ for all $n$.
For this reason, we need a new type of \emph{integrable random variables}.
\begin{definition}[Integrable random variable]
Let $(\tpX,d)$ be a metric space.
The type of \emph{integrable random variables} is the effective completion of the type of continuous random variables under the metric
\begin{equation} \label{eq:randomvariableuniformmetric} d_1(X,Y) = \int_\Omega d(X(\omega),Y(\omega)) dP(\omega) . \end{equation}
\end{definition}
\noindent
If $d$ is a bounded metric, then this metric is equivalent to the Fan metric.

For integrable random variables taking values in the reals, the expectation is defined in the usual way:
\begin{definition}[Expectation]
If $X:\Omega \fto \tpRe$ is a continuous real-valued random variable, the \emph{expectation} of $X$ is given by the integral
\[ \textstyle \Ex(X) = \int_{\Omega} X(\omega) \, dP(\omega) , \]
which always exists since $X$ has compact values.
\par 
If $X:\Omega \mfto \tpRe$ is an integrable real-valued random variable, and $X$ is presented as $\lim_{n\to\infty}X_n$ for some sequence of random variables satisfying $\Ex[|X_{n_1}-X_{n_2}|]\leq 2^{-\min(n_1,n_2)}$, define
\[ \textstyle \Ex(X) = \lim_{n\to\infty} \Ex(X_n), \]
which is an effective Cauchy sequence since $|\Ex(X_{n_1})-\Ex(X_{n_2})| \leq \Ex(|X_{n_1}-X_{n_2}|)$.
\end{definition}
\noindent

We can effectivise Lesbegue spaces $\mathcal{L}^p(\tpX)$ of integrable random variables through the use of effective Cauchy sequences in the natural way:
If $(\tpX,|\cdot|)$ is a normed space, then the type of $p$-integrable random variables with values in $\tpX$ is the effective completion of the type of $p$-integrable continuous random variables under the metric $d_{p}(X,Y)=||X-Y||_p$ induced by the norm
\begin{equation} \label{eq:randomvariableuniformnorm} \textstyle ||X||_{p} = \left( \strut \int_{\Omega} | X(\omega)|^p\,dP(\omega) \right)^{1/p} = \bigl(\Ex(|X|^p)\bigr)^{1/p}. \end{equation}
We can easily prove the Cauchy-Schwarz and triangle inequalities for measurable random variables
\( ||XY||_{pq/(p+q)} \leq ||X||_{p}  \cdot ||Y||_{q} \quad \text{and} \quad ||X+Y||_{p} \leq ||X||_{p}+||Y||_{p} \,. \)

%
%If $X$ is a real-valued random variable, then the \emph{cumulative distribution} of $X$ is the lower-semicontinuous function $f_X:\R \fto [0,1]$ given by
%\( f_X (x) = \Pr(X \in (x,\infty)) . \)

%
%The conditional distribution of $X$ given $Y$ is
%\[ \Pr(X\in U \mid Y \in  V) = \frac{\Pr( X^{-1}(U) \cap Y^{-1}(V) )}{\Pr( Y^{-1}(V) )} . \]

%Random variables $X_1$ and $X_2$ are \emph{independent} if for all open sets $U_1$, $U_2$, $\Pr(X_1\in U_1 \wedge X_2\in U_2) = \Pr(X_1\in U_1) \cdot \Pr(X_2\in U_2)$.
%If $X_1$ is a random variable on $(\Omega_1\times \Omega_2)$, we say $X_1$ is independent of $\Omega_2$ if $X_1$ is independent of the random variable $X_2$ defined by $X_2(\omega_1,\omega_2)=\omega_2$.
%In particular, if $X_1$ is a random variable on $(\Omega_1\times \Omega_2)$ and we can write $X_1(\omega_1,\omega_2)=X_1(\omega_1,\omega_2')$ for all $\omega_2,\omega_2'$, then $X_1$ is independent of $\Omega_2$.

\begin{theorem}[Expectation]
\label{thm:expectation}
Let $X$ be a positive real-valued random variable such that $\Ex(X)<\infty$.
Then
\begin{equation*} \textstyle \Ex(X) = \int_{0}^{\infty} \Pr(X>x) dx = \int_{0}^{\infty} \Pr(X \geq  x) dx . \end{equation*}
\end{theorem}
\noindent
Note that the first integral is computable in $\tpR^+_<$, but the second integral is in general uncomputable in $\tpR^+_>$, due to the need to take the limit as the upper bound of the integral goes to infinity.
However, the second integral may be computable if the tail is bounded, for example, if $X$ takes bounded values.
The proof follows from the definition of the lower integral:
\begin{proof}
First assume $X$ is a continuous random variable, so by definition, $\Ex(X) = \int_\Omega X(\omega)\,dP(\omega)$.

The definition of the lower horizontal integral gives
$\int_\Omega X(\omega)\,dP(\omega) \geq \sum_{i=0}^{n-1} (x_i-x_{i-1}) P(\{\omega\mid X(\omega)>x_i\})$
for all values $0=x_0<x_1<\cdots<x_n$.
Take $x_{i}-x_{i-1}<\epsilon$ for all $i$.
Then $\Ex(X)+\epsilon = \int_\Omega X(\omega)+\epsilon\,dP(\omega)  \geq \sum_{i=1}^{n} (x_i-x_{i-1}) P(\{\omega\mid X(\omega) +\epsilon > x_i\}) = \sum_{i=1}^{n} \int_{x_{i-1}}^{x_i} \Pr(X(\omega) > x_i-\epsilon) \,dx  \geq \sum_{i=1}^{n} \int_{x_{i-1}}^{x_i} \Pr(X(\omega) > x_{i-1}) \,dx \geq \sum_{i=1}^{n} \int_{x_{i-1}}^{x_i} \Pr(X(\omega) > x) \, dx = \int_{0}^{x_{n-1}} \Pr(X(\omega) > x)$.
Taking $n\to\infty$ gives $\Ex(X) \geq \int_{0}^{\infty}\Pr(X>x)dx - \epsilon$, and since $\epsilon$ is arbitrary, $\Ex(X) \geq \int_{0}^{\infty} \Pr(X > x) dx$.

The definition of the lower horizontal integral gives for all $\epsilon>0$, there exist $0=x_0<x_1<\cdots<x_n$, such that $\int_\Omega X(\omega)\,dP(\omega) \leq \sum_{i=1}^{n} (x_i-x_{i-1}) P(\{\omega\mid X(\omega)>x_i\})+\epsilon$.
By refining the partition if necessary, we can assume $x_i-x_{i-1}<\epsilon$ for all $i$. 
Then $\Ex(X)-\epsilon \leq \sum_{i=1}^{n} (x_i-x_{i-1}) P(\{\omega\mid X(\omega)>x_{i}\}) = \sum_{i=1}^{n} \int_{x_{i-1}}^{x_{i}} P(\{\omega\mid X(\omega)>x_{i}\} dx \leq \sum_{i=1}^{n} \int_{x_{i-1}}^{x_{i}} P(\{\omega\mid X(\omega)>x\} dx=\int_{0}^{x_n} P(\{\omega\mid X(\omega)>x\} dx\leq\int_{0}^{\infty} P(\{\omega\mid X(\omega)>x\} dx$.
Hence $\Ex(X) \leq \int_{0}^{\infty}\Pr(X>x)dx + \epsilon$, and since $\epsilon$ is arbitrary, $\Ex(X) \leq \int_{0}^{\infty} \Pr(X > x) dx$.

The case of measurable random variables follows by taking limits.

We show $\int_{0}^{\infty}\Pr(X\geq x)dx = \Ex(X)$ since  $\int_{0}^{\infty}\Pr(X > x)dx \leq \int_{0}^{\infty}\Pr(X\geq x)dx \leq \int_{0}^{\infty}\Pr(X + \epsilon > x)dx = \epsilon + \int_{0}^{\infty}\Pr(X  \geq x)dx$  for any $\epsilon>0$.
\end{proof}
\noindent
By changing variables in the integral, we obtain:
\begin{corollary}
If $X$ is a real-valued random variable, then for any $\alpha\geq 1$,
\[ \textstyle \Ex(|X|^\alpha) = \int_{0}^{\infty} \alpha \, x^{\alpha-1} \, \Pr(X>x) dx = \int_{0}^{\infty} \alpha x^{\alpha-1} \Pr(X \geq  x) dx . \]
\end{corollary}

\begin{remark}[Expectation of a distribution]
Theorem~\ref{thm:expectation} shows that the expectation of a random variable depends only on its \emph{distribution}.
Indeed, we can define the expectation of a probability \emph{valuation} $\pi$ on $[0,\infty[$ by
\[ \textstyle \Ex(\pi) = \int_{0}^{\infty} \pi(\,]x,\infty[\,) dx  = \int_{0}^{\infty} \pi(\,[x,\infty[\,) dx  . \]
\end{remark}
\noindent
If $f:\tpX\fto\tpRe^+_<$, then we can compute the \emph{lower expectation} of $f(X)$ by
\begin{equation}\label{eq:lowerexpectation} \textstyle \Ex_<(f(X)) := \int_{0}^{\infty} \Pr\bigl(X\in f^{-1}(\,]\lambda,\infty[\,)\bigr) d\lambda . \end{equation}

%Unfortunately, for $f:\tpX\fto\tpRe^+_>$, the corresponding \emph{upper expectation} of $f(X)$ cannot in general be computed by
%\[ \textstyle \Ex_>(f(X)) := \int_{0}^{\infty} \Pr\bigl(X\in f^{-1}([\lambda,\infty[)\bigr) d\lambda  \]
%since taking the upper limit in the integral is not continuous into $\tpRe^+_>$.
%However, if $f$ is effectively bounded, then the integral can be computed.
%\marginpar{Where should this go?}

We have an effective version of the classical dominated convergence theorem.
\begin{theorem}[Dominated convergence]
Suppose $X_n\to X$ weakly, and there is an integrable function $Y:\Omega \fto \tpRe$ such that $|X_n| \leq Y$ for all $n$ (i.e. $\Pr(Y-|X_n|\geq0)=1$) and that $\Ex|Y|<\infty$.
Then $X_n$ converges effectively under the metric~\eqref{eq:randomvariableuniformmetric}.
In particular, the limit of $\Ex(X_n)$ always exists
\end{theorem}
\begin{proof}
Since $\Ex(Y) < \infty$, the probabilities $\Pr(Y \geq y) \to 0$ as $y\to\infty$. 
For fixed $\epsilon>0$, let $b(\epsilon)=\sup\{ y\mid \Pr(Y\geq y) \geq \epsilon$, which is computable in $\R_>$ given $\epsilon$.
Then $\sup\{ \int_A Y dP \mid P(A)\leq \epsilon \} \leq \int_{b(\epsilon)}^{\infty} \Pr(Y\geq y) dy=\Ex(Y)-\int_{0}^{b(\epsilon)} \Pr(Y > y) dy$ in $\R^+_>$.
For continuous random variables $X_m$, $X_n$ with $2^{-m},2^{-n}<\epsilon$, taking $A_\epsilon=\{\omega\mid d(X_m(\omega),X_n(\omega))\geq \epsilon\}$ gives 
\( \Ex(|X_m-X_n|) \leq \epsilon + \int_{A_\epsilon}|X_m(\omega)-X_n(\omega)|\,dP(\omega) \leq \epsilon + \int_{A_\epsilon}|X_m(\omega)|+|X_n(\omega)|dP(\omega)  \leq \epsilon + \int_{A_\epsilon} 2|Y| dP \leq \epsilon + 2 \int_{b(\epsilon)}^{\infty} \Pr(Y\geq y) dy, \)
which converges effectively to $0$ as $\epsilon\to 0$.
\end{proof}

%\begin{definition}[Measurable events]
%Any real-valued random variable $X$ such that $P(X\in\{0,1\})=1$ is a \emph{characteristic} function.
%A subset $W$ of $\Omega$ is \emph{measurable} if it has a measurable characteristic (indicator)  function
%\end{definition}

\subsection{Independence and Conditioning}

\newcommand{\eval}{\mathrm{eval}}

The concept of conditional random variable is subtle even in classical probability theory.
The basic idea is that if we condition a random quantity $Y$ on some information of kind $\calX$, then we can reconstruct $Y$ given a value $x$.
Classically, conditional \emph{random variables} are \emph{not} defined, but conditional distributions and expectations are.
Conditional expectations can be shown to exist using the Radon-Nikodym derivative, but this is uncomputable~\cite{HoyrupRojasWeihrauch2011}.

In the classical case, we condition relative to a sub-sigma-algebra of the measure space.
In the computable case, it makes sense to consider instead a sub-topology $\calT$ on $\Omega$.
We first need to define concepts of $\calT$ measurability and $\calT$ independence
\begin{definition}
Let $\calT$ be a topology on $\Omega$. 
We say a measurable random variable $X$ is $\calT$-measurable if $X$ is a limit of $\calT$-continuous random variables $X_m$, and write $\tpRV_\calT(\tpX)$ for the type of $\calT$-measurable random variables with values in $\tpX$.
\end{definition}
We define independence relative to a sub-topology using the identity random variable $I:\Omega\fto\Omega$.
\begin{definition}[Independence]
\label{defn:independentrandomvariables}
Given a topology $\calT$ on $\Omega$, we say a random variable $X:\Omega\mfto\tpX$ is \emph{independent} of $\calT$ if
\begin{equation} \Pr(X\times I \in U \times W) = \Pr(X \in U) P(W) \end{equation}
whenever $U\in\opset(\tpX)$ and $W\in\calT$. We write $\tpRV_{\perp\calT}(\tpX)$ for the type of $\calT$-independent random variables with values in $\tpX$.

We say random variables $X_1,\ldots,X_k$ are \emph{jointly} independent of $\calT$ if the product $\prod_{i=1}^{k}X_i$ is independent of $\calT$.

We say random variables $X_1,X_2$ taking values in $\tpX_1$, $\tpX_2$ are independent if for all open $U_1\subset \tpX_1$ and $U_2\subset \tpX_2$, we have
\begin{equation}
 \Pr(X_1\in U_1 \wedge X_2\in U_2) = \Pr(X_1\in U_1) \cdot \Pr(X_2\in U_2) .
\end{equation}
\end{definition}
\noindent
If $X:\Omega\fto\tpX$ is continuous, the definition of independence with respect to $\calT$ reduces to
\[ P(\{\omega\in\Omega \mid X(\omega)\in U \wedge \omega\in W) = \P(X \in U) P(W) \]
If $X$ is independent of $\calT$, and $\calT$ is the topology generated by a continuous random variable $Y$, then $X$ is independent of $Y$, since
\[ \Pr(X\in U\wedge Y\in V) = \Pr(X\in U \wedge I\in Y^{-1}(V)) = \Pr(X \in U) P(Y^{-1}(V)) = \Pr(X\in U) \Pr(Y\in V) . \]
Note that it is possible for $X_1$, $X_2$ to be independent of $\calT$, but $X_1\times X_2$ not to be.
Clearly if $X_1$ and $X_2$ are independent real-valued integrable random variables, then $\Ex(X_1X_2)=\Ex(X_1)\Ex(X_2)$.

%Since $\Omega$ is canonically isomorphic to $\Omega\times\Omega$, given random variables $X_1$, $X_2$ taking values in $\tpX_1$, $\tpX_2$, we can always construct independent random variables $Y_{1,2}$ with the same distribution as $X_{1,2}$ by taking $Y_i(\omega_1,\omega_2) = X_i(\omega_i)$ for $i=1,2$.
%We call the random variable $Y_1\times Y_2$ an \emph{independent product} of $X_1$ and $X_2$.
%Countable independent products can be constructed since $\Omega$ is isomorphic to $\Omega^\infty$.
%We say random variables $X_1$ and $X_2$ are \emph{strongly independent} if they are formed as an independent product.

\newcommand{\pts}{\mathrm{Pt}}

Since the space $(\tpT,\calT)$ need not be a Kolmogorov ($T^0$) space, it is useful to quotient out sets of points which cannot be distinguished from each other by the topology $\calT$.
The resulting \emph{quotient space} $\tpT/\calT$ is determined by the equivalence relation 
\[ t_1 \sim_\calT t_2 \iff \forall U\in\calT,\ t_1\in U \iff t_2\in U , \]
and quotient map $q_\calT:\tpT\fto \tpT/\calT$ is computable.
It is easy to see that the quotient space is Hausdorff if, and only if,
\[ t_1 \sim_\calT t_2 \vee \exists U_1,U_2\in \calT, t_1\in U_1\wedge t_2\in U_2\wedge U_1\cap U_2=\emptyset , \]
and that if the quotient space is Hausdorff, then we can define a metric by
\[ \begin{aligned} & d_\calT (q_\calT(t_1),q_\calT(t_2)) = \sup \{ d(U_1,U_2) | U_1\ni t_1 \wedge U_2\ni t_2\} \\ &\qquad\qquad\text{ where } d(U_1,U_2) = \inf \{ d(s_1,s_2) | s_1\in U_1 \wedge s_2 \in S_2\} .  \end{aligned} \]
Further, if $\calT$ is the preimage of a Hausdorff topology by a continuous function, then the quotient space $\tpT/\calT$ is Hausdorff.

We now give our notion of conditional random variable $Y|\calX$, where $\calX$ is a \emph{topology} on $\Omega$.
\begin{definition}[Conditional random variable]
\label{defn:conditionalrandomvariable}
Let $\mathcal{X}$ be a topology on $\Omega$ such that the quotient space $\Omega/\calX$ is a Polish space.
A \emph{random variable taking values in $\tpY$ conditional on $\calX$} is a continuous function $Y|\calX:(\Omega,\calX)\fto\tpRV(\tpY)$ such that $Y|\calX(\omega)$ is independent of $\calX$ for all $\omega\in\Omega$.
\end{definition}
\noindent
Note that we do not require that the values of $Y|\calX$ be \emph{jointly} independent of $\calX$.
%Also, the restriction that $\Omega/\calX$ is Hausdoff gives no loss of generality, since the preimage of a Hausdorff topology under a continuous map has a Hausdorff quotient.

It is convenient to allow a different space for the conditioning variable and consider functions $Y| : \tpX\fto\tpRV_{\perp\calX}(\tpY)$, and write $Y|x$ for $Y|(x)$.
The idea of conditioning is that knowing the value of a $\calX$-measurable random variable in $\tpX$, we know the random variable $Y$.
\begin{proposition}\label{prop:conditionalrandomvariable}
Let $\calX$ a topology on $\Omega$. The operator $\tpRV_\calX(\tpX)\times (\tpX\fto\tpRV_{\perp\calX}(\tpY)) \fto \tpRV(\tpY)$ extending the operator on continuous random variables $(\Omega\fto\tpX)\times(\tpX\times\Omega\fto\tpY):(X,Y|) \mapsto Y|(X(\omega),\omega)$, is computable.
\end{proposition}
\begin{proof}
For the case that each $Y|x$ is a continuous random variable, note that $Y| : \tpX\fto(\Omega\fto\tpY) \equiv \tpX\times\Omega\fto\tpY$.

Suppose $X$ is a $\calX$-measurable simple random variable.
Let $U\in\calX$, and let $x_1,\ldots,x_{j}$ be the values of $X$ lying in $U$.
Then \( \Pr(X\in U \wedge Y\in V) = \sum_{i=1}^{j} \Pr(X = x_i \wedge [Y|x_i]\in V) . \)
Since each $Y|x_i$ is independent of $\calX$, \( \sum_{i=1}^{j} \Pr(X = x_i \wedge [Y|x_i]\in V) = \sum_{i=1}^{j} \Pr(X_m=x_i) \Pr([Y|x_i]\in V) = \int_{\omega\in U} \Pr([Y|X(\omega)]\in V) d\Pr[X](\omega). \)

Since measurable random variables are limits of fast converging Cauchy sequences of continuous random variables, we have $X:\N\times\Omega\fto\tpX$ and $Y|:\tpX\times\N\times\Omega\fto\tpY$; further, we can assume each $X_m$ is a simple random variable.
Denote $X(m,\cdot)$ by $X_m$, $Y|(\cdot,n,\cdot)$ by $Y_n|$, $Y|(x,n,\cdot)$ by $Y_n|x$, and $Y|(X_m(\omega),n,\omega)=[Y_n|X_m](\omega)$.

For fixed $m$, $X_m$ takes finitely many values $x_{m,1},\ldots,x_{m,k_m}$, on clopen sets $W_{m,1},\ldots,W_{m,k_m}$ in $\calX$.
Since for for all $x$, and for $n_1,n_2\geq n$, $d(Y_{n_1}|x,Y_{n_2}|x)<2^{-n}$, we have $d(Y_{n_1}|X_m,Y_{n_2}|X_m)<k_m 2^{-n}$, so $(Y_n|X_m)_{n\in\N}$ is an effective Cauchy sequence, and converges to a random variable $Y|X_m$.

By the argument of the proof of Theorem~\ref{thm:imagerandomvariable}, for any $n$ there exists $M(n)$ such that $\Pr(d(Y|X_{m_1},Y|X_{m_2}) > 2^{-n})<2^{-n}$ whenever $m_1,m_2\geq M(n)$, so $(Y|X_m)_{m\in\N}$ converges effectively.
\end{proof}

\begin{theorem}\label{thm:conditionalrandomvariable}
Let $\calX$ a topology on $\Omega$ such that the quotient space $\Omega/\calX$ is a Polish space.
Then there is a computable embedding of $(\Omega/\calX \fto\tpRV_{\perp\calX}(\tpY)) \hookrightarrow \tpRV(\tpY)$ extending the operator on continuous random variables given by $Y|\calX \mapsto Y|\calX(\omega)(\omega)$.
\end{theorem}
\begin{proof}
Apply Proposition~\ref{prop:conditionalrandomvariable} where $\tpX=\Omega/\calX$ and $X$ is the quotent map.
\end{proof}

%\begin{proposition}
%Let $\calX$ a topology on $\Omega$ such that the quotient space $\Omega/\calX$ is a Polish space.
%Then there the computable embedding of $(\Omega/\calX \fto\tpRV_{\perp\calX}(\tpY)) \hookrightarrow \tpRV(\tpY)$ given by Theorem~\ref{thm:conditionalrandomvariable} does not have a computable inverse.
%\end{proposition}

The following result is an analogue of classical results on conditioning.
\begin{proposition}
Suppose $Y|\calX:\Omega/\calX \fto \tpRV_{\perp\calX}(\tpY)$ and $Y:\tpRV(\tpY)$ the unconditioned version given by Theorem~\ref{thm:conditionalrandomvariable}.
$Y$ is $\calX$-measurable if, and only if, for all $\omega$, $Y|\calX(\omega)$ is a constant random variable.
$Y$ is independent of $\calX$ if, and only if, $Y|\calX(\omega_1)=Y|\calX(\omega_2)$ for all $\omega_1,\omega_2\in\Omega$.
\end{proposition}

We cannot compute $Y$ given a conditional random variable $Y|\calX$, and a $\calX$-measurable random variable $X$ taking values in $\tpX$, since the value of $X$ might not be sufficient to identify a unique element of $\calX$. 
However, we can compute the joint distribution of $X\times Y$ using conditional probabilities.
We define these in terms of conditional random variables, and show that they satisfy the usual classical properties.
\begin{definition}
The \emph{conditional probability} $\Pr(Y|\calX)$ is the function $(\Omega,\calX) \fto \tpPr(\tpY)$ defined by
\begin{equation}\label{eq:conditionalprobability} \Pr(Y|\calX) : (\Omega,\calX)\times\tpOp(\tpY)\fto \tpIv_<: (\omega,V) \mapsto \Pr([Y|\calX(\omega)]\in V) . \end{equation}
We define $\Pr(Y\in V|\calX):(\Omega,\calX)\fto\tpIv_<$, $\Pr(Y|x):\tpOp(\tpY)\fto\tpIv$ and $\Pr(Y\in V|x):\tpIv_<$ in the natural way.
\end{definition}
We now show that the joint distribution of $\calX$-measurable $X$ and $Y$ can be computed from $X$ and $Y|\calX$.
As a consequence of Theorems~\ref{thm:randomvariabledistribution} and~\ref{thm:productrandomvariable}, we can define a distribution conditional on a random variable lying in a given set:
\begin{definition}
Let $X:\tpRV(\tpX)$ be a random variable, and $U\in\tpOp(\tpX)$.
Define the induced valuation on $\Omega$ by 
\[ P[\cdot|X\in U](V) = \Pr(X\times I \in U\times V) . \]
Write $P(V|X\in U)=P[\cdot|X\in U](V)$.
\end{definition}
\noindent 
Note that $P[\cdot|X\in U]$ has total measure $P(U)$, and if $X$ is continuous, $P[\cdot|X\in U](V) = P(X^{-1}(U)\cap V)$.

Further, if $X$ is $\calX$-measurable, then the projection $\Omega\fto\Omega/\calX$ induces a measure on $\Omega/\calX$.
\begin{proposition}
\label{prop:conditionalprobability}
If $X:\tpRV(\tpX)$ is a $\calX$-measurable, and $Y|\calX:(\Omega,\calX)\fto\tpRV(\tpY)$ has values which are independent of $\calX$, then 
\[ P(X\in U \wedge Y\in V) = \int_{[\omega]\in\Omega/\calX} \Pr(Y|\calX[\omega]\in V) dP([\omega]|X\in U) . \]
\end{proposition}
\begin{proof}
The result clearly holds for simple random variables $X$, and extends to all random variables.
\end{proof}

Note that if $X:\tpRV(\tpX)$, and $Y|\tpX:\tpX\fto\tpRV(\tpY)$, then the distributions $\xi=\Pr[X]$ and $\eta_x=\Pr[Y|x]$ are computable as valuation.
Define the joint distribution $\xi\!\rtimes\!\eta$ on $\tpX\times\tpY$ by its integrals
\begin{equation}
  \int_{\tpX\times\tpY} \psi(x,y) d\xi\!\!\rtimes\!\!\eta(x,y) = \int_{\tpX} \int_{\tpY} \psi(x,y) d\eta_x(y) \, d\xi(x)
\end{equation}
In particular $\xi\!\!\rtimes\!\!\eta(U\times V) = \int_{x\in U} \eta_x(V) d\xi(x)$.
Proposition~\ref{prop:conditionalprobability} strengthens this result by weakening the requirement that $Y$ is defined on $\tpX$ itself.

\begin{definition}
Suppose each $Y|\calX[\omega]$ is an integrable random variable. 
The \emph{conditional expectation} $\Ex(Y|\calX)$ is the function $(\Omega,\calX) \fto \tpR$ defined by
\begin{equation}\label{eq:conditionalexpectation} \Ex(Y|\calX) : \omega \mapsto \Ex(Y|\calX(\omega)). \end{equation}
\end{definition}
Then for any random variable $X:\tpRV(\tpX)$, we have $\Ex[Y|X]:\tpRV(\tpR)$ by composition.
\begin{proposition}
If $Y|\calX:(\Omega,\calX):\tpRV(\tpY)$ is independent of $\calX$, then for any $\calX$-measurable $X:\tpRV(\tpR)$, we have
\begin{equation} \Ex(X\Ex(Y|\calX)) = \Ex(XY) . \end{equation}
In particular, $\Ex(Y) = \Ex(\Ex(Y|\calX))$.
\end{proposition}
\begin{proof}
If $Y|\calX$ takes finitely many values, each on sets $A_i \in \calX$, each value is continuous, and $X$ is simple taking values $x_i$ on $A_i$, then
\[ \begin{aligned} \Ex(XY) &\textstyle = \int_{\omega\in\Omega} X(\omega) Y|\calX(\omega) (\omega) dP(\omega) = \int_{\omega\in\Omega} \sum_{i=1}^{k} I[\omega\in A_i]  x_i Y|\calX(A_i)(\omega) dP(\omega) \\ &\qquad\textstyle = \sum_{i=1}^{k} x_i \int_{\omega\in A_i} Y|\calX(A_i)(\omega) dP(\omega) = \sum_{i=1}^{k} P(A_i) x_i \Ex(Y|\calX(A_i)) = \Ex(X\Ex(Y|\calX)) . \end{aligned} \] 
The result follows by extension to measurable random variables.
\end{proof}

In the definition of conditional random variable, we use objects of type $\tpX\fto\tpRV(\tpY)$, which are random-variable-valued functions, rather than \emph{random functions} with type $\tpRV(\tpX\fto\tpY)$.
The latter type encodes strictly more information than the former.
\begin{theorem}[Random function]
The natural bijection $\tpRV(\tpX\fto\tpY) \hookrightarrow (\tpX\fto\tpRV(\tpY))$ is computable, but its inverse is not continuous.
\end{theorem}
\begin{proof}
For fixed $x$, evaluation $\varepsilon_x:(\tpX\fto\tpY)\fto\tpY:f\mapsto f(x)$ is computable, so by Theorem~\ref{thm:imagerandomvariable}, $
\varepsilon(F):\tpRV(\tpY)$ is computable for any $F:\tpRV(\tpX\fto\tpY)$ given $x$. Hence the function $x\mapsto \varepsilon_x(F)$ is computable.

Conversely, let $X=\{0,1\}^\omega$ and $Y=\{0,1\}$.
Define $F(x,\omega,n)=1$ if $x|_n=\omega|_n$, and $0$ otherwise.
Then for fixed $x$, $F(d(x,\cdot,n),0)=2^{-n}$, so $F(x,\cdot,n)$ converges to $0$ uniformly in $x$.

For fixed $\omega$, $d(F(\cdot,\omega,n_1),F(\cdot,\omega,n_2)) = \sup_{x\in X} d(F(x,\omega,n_1),F(x,\omega,n_2)) = 1$, since (for $n_1<n_2$) there exists $x$ such that $x|_{n_1}=\omega|_{n_1}$ but $x|_{n_2}\neq\omega|_{n_2}$.
Hence $d(F(\cdot,\cdot,n_1),F(\cdot,\cdot,n_2))=1$ for all $n_1,n_2$, and the sequence is not a Cauchy sequence in $\tpRV(\tpX\fto\tpY)$.
\end{proof}

\section{Discrete-Time Stochastic Processes}

A \emph{discrete-time stochastic process} with state space $\tpX$ is a random variable $\seq{X}=(X_0,X_1,X_2,\ldots)$ taking values in $\tpX^\infty$.
%which is \emph{non-anticipative}, i.e. $X_{n+1}$ has a conditional value $X_{n+1}|X_n,\ldots,X_0$.
A \emph{Markov process} is a stochastic process such that $X_{n+1}$ depends only on the previous state $X_n$, so is determined by the conditional value $X_{n+1}|X_n$, such that $X_{n+1}|X_n=x_n$ is independent of $(X_0,X_1,\ldots,X_{n-1})$.
A Markov process is \emph{stationary} if the distributions of $X_{n+1}|X_n$, i.e. $\Pr\bigl((X_{n+1}|X_{n}=x_n)\in U\bigr)$, are equal.
In this case we can write $\Pr(X_{n+1}|X_n) = F_n:\tpX \fto(\Omega\mfto\tpX)$, where $F_n(x,\omega)=F_n(\omega_0,\omega_1,\ldots)=F(\omega_n)$.
Hence the process is defined by $F:\tpX \fto \tpRV(\tpX)$.

Typically, we are only interested in the \emph{distribution} of the states $X_n$, and so rather than treating $X_n$ as a random variable $X_n:\Omega\mfto\tpX$, we consider $X_n\in \tpPr(\tpX)$.
Then the Markov process is defined by $F:\tpX \fto \tpPr(\tpX)$.

When working in Cartesian-closed categories, objects of the form $(\tpX\fto \tpT)\fto \tpT$ for some fixed type $\tpT$ are an example of a \emph{monad}~\cite{Street1972}.
They support standard manipulations which make them ideal for the representation of dynamic systems.
When $\tpT=\tpSi$, the Sierpinski type, we obtain categories of overt and compact sets~\cite{Escardo2004}, which form a basis for discrete-time nondeterministic systems~\cite{Collins2009CIE}.
Since $\tpPr(\tpX)$ is a subtype of $(\tpX\fto \tpHl)\fto \tpHl$ we can take $\tpT=\tpHl$ and obtain the same operators for discrete-time stochastic systems.

\begin{proposition}\label{prop:canonicaloperators}\mbox{}
Let $\tpT$, $\tpX$ and $\tpY$ be elements of the category of computable types.
Then the following operators are computable:
\begin{enumerate}
 \item The embedding of $\tpX$ in $(\tpX\fto\tpT)\fto\tpT$ given by $\delta_x(\phi) = \phi(x)$ for $x\in\tpX$ and $\phi:\tpX\fto\tpT$.
 \item The canonical equivalence between $\tpX\fto((\tpY\fto\tpT)\fto\tpT)$ and $(\tpY\fto\tpT)\fto(\tpX\fto\tpT)$ given by \( F^*\psi(x) = F(x)(\psi) \) for $F:\tpX\fto((\tpY\fto\tpT)\fto\tpT)$ and $\psi:\tpY\fto\tpT$.
 \item An element $f$ of $\tpX\fto\tpY$ lifts to an operator $f_*$ from $(\tpX\fto\tpT)\fto\tpT$ to $((\tpY\fto\tpT)\fto\tpT)$ defined by
\( f_*\mu(\psi) = \mu(\psi\circ f) \) for $\mu:(\tpX\fto\tpT)\fto\tpT$ and $\psi:\tpY\fto\tpT$.
 \item An element $F$ of $\tpX\fto((\tpY\fto\tpT)\fto\tpT)$ lifts to an operator $F_*$ from $(\tpX\fto\tpT)\fto\tpT$ to $((\tpY\fto\tpT)\fto\tpT)$ defined by
\( F_*\mu(\psi) = \mu(\lambda x.\,F(x)(\psi)) . \)
 \item Given $F:(\tpX\fto\tpT)\fto\tpT$ and $G:(\tpY\fto\tpT)\fto\tpT$, the skew-products $F \rtimes G:(\tpX\times\tpY\fto\tpT)\fto\tpT$ defined by
\( (F \rtimes G)(\psi) = F( \lambda x.G(\lambda y.\psi(x,y)) ) . \) and \( (F \ltimes G)(\psi) = G( \lambda y.F(\lambda x.\psi(x,y)) ) . \)
  \par We write $F\times G$ for the product if $F \rtimes G = F \ltimes G$ for all $F,G$ in some restricted class of interest.
\end{enumerate}
\end{proposition}
For the case of set types, the embedding $\tpX\hookrightarrow((\tpX\fto \tpT)\fto \tpT)$ a singleton set; for measures, the point-measure $\delta_x$.
Note that if $F:(\tpX\fto\tpT)\fto\tpT$ and $G:(\tpY\fto\tpT)\fto\tpT$, then in general $F \rtimes G$ and $F \ltimes G$ are not equal.
Equality (i.e. commutativity of the product) does hold in many important cases, including products of measures.
The generalisation to \emph{monads} $\tpMon(\tpX)$ requires canonical operators $\tpX \fto \tpMon(\tpX)$ and $(\tpX\fto\tpMon(\tpY)) \fto (\tpMon(\tpX)\fto\tpMon(\tpY))$.

We now apply the standard push-forward operators of Proposition~\ref{prop:canonicaloperators} to the case of probability measures.
Computability of the operators on $(\tpX\fto\tpHl)\fto\tpHl$ is clear, it remains to check the linearity properties and the unit total measure.

\begin{lemma}
\label{lem:pointmeasure}
There is a computable point-measure operator taking $x\in \tpX$ to $\delta_x \in \tpPr(\tpX)$.
\end{lemma}
\begin{proof}
For $\psi:\tpX\fto \tpHl$, define $\delta_x(\psi) = \psi(x)$.
Then $\delta_x(\alpha_1\psi_1+\alpha_2\psi_2)=(\alpha_1\psi_1+\alpha_2\psi_2)(x)=\alpha_1\psi_1(x)+\alpha_2\psi_2(x)=\alpha_1\delta_x(\psi_1)+\alpha_2\delta_x(\psi_2)$, and if $\psi\equiv 1$, then $\delta_x(\psi) = 1$, so $\delta_x$ is a probability measure.
\end{proof}

\begin{proposition}
\label{prop:pushforward}
There is a computable push-forward operator taking a function $F:\tpX\fto \tpPr(\tpY)$ and $\mu \in \tpPr(\tpX)$ to the push-forward distribution $F_*\mu \in \tpPr(\tpY)$ is computable.
\end{proposition}
\begin{proof}
For $\psi:\tpY\fto \tpT$, we have $F_*\mu(\psi) = \mu(\lambda x.F(x)(\psi))$ is computable. We need to check that $F_*\mu$ is a probability measure.
It is easy to verify that $F^*(\alpha_1\psi_1+\alpha_2\psi_2) = \alpha_1F^*(\psi_1)+\alpha_2F^*(\psi_2)$, and hence $F_*\mu(\alpha_1\psi_1+\alpha_2\psi_2)=\alpha_1 F_*\mu(\psi_1) + \alpha_2 F_*\mu(\psi_2)$.
If $\psi \equiv 1$, then $\phi = F^*\psi = \lambda x.F(x)(\psi) \equiv 1$ since for all $x\in \tpX$, $\phi(x)=F(x)(\psi)$ and $F(x)$ is a probability measure.
Then $\mu(\phi)=1$ as $\mu$ is a probability measure.
\end{proof}

\begin{corollary}
If $f:\tpX\fto \tpY$, then $f$ induces a computable operator $f_*:\tpPr(\tpX)\fto\tpPr(\tpY)$ by $f_*\mu = F_*\mu$ where $F(x) = \delta_{f(x)}$.
Explicitly, $f_*\mu(\psi) = \mu(\psi\circ f)$ for $\psi:\tpY\fto\H$.
\end{corollary}

\begin{proposition}\label{prop:pushforwardjoint}
The push-forward operator taking a function $F:\tpX\fto\tpPr(\tpY)$ and probability measure $\mu\in\tpPr(\tpX)$ to the joint distribution $(\mu,F_*\mu)=(\id\rtimes F)_*\mu$ on $\tpX\times\tpY$ is computable
\end{proposition}
\begin{proof}
We have $F:\tpX\fto ( (\tpY\fto \tpHl) \fto \tpHl)$ and $\mu:( (\tpX\fto \tpHl) \fto \tpHl)$.
Define $(\id \rtimes F)$ to be the function $\tpX\fto\tpPr(\tpX\times \tpY)$ given by $(\id \rtimes F)(x)(\psi) = F(x)(\lambda y.\psi(x,y))$.
Note that if $\psi:\tpX\times \tpY\fto \tpHl$ is the constant function $1$, then $\lambda y.\psi(x,y) \equiv 1$, so $(\id \rtimes F)(x)(\psi) = F(x)(\phi) = 1$ since $F(x)$ is a probability distribution.
Then by Proposition~\ref{prop:pushforward}, $(\id\times F)_*\mu$ is computable in $\tpPr(\tpX\times \tpY)$.
\end{proof}

We first consider the simplest approach to stochastic processes, where we only compute the \emph{distribution} of the states.
A Markov process is then defined by a stochastic update rule $F$ for states of a dynamic system.
Given $x\in \tpX$, the probability distribution of the next state is $F(x)$.
Denoting the state at time $n$ by a random variable $\tpX_n$, a Markov process can be written $F(x)(U) = \Pr( \tpX_{n+1} \in U \mid \tpX_n = x \}$.

\begin{definition}
The type of simple Markov processes on a type $\tpX$ is $\tpX \fto \tpPr(\tpX)$.
\end{definition}
Since a continuous function $f:\tpX\fto \tpY$ induces a natural operator $F:\tpX\fto \tpPr(\tpY)$ by $F(x)(\psi) = \psi(f(x))$ for $\psi:\tpY\fto \H$, any deterministic system can be seen as a stochastic system.

%A push-forward operator $f_*:\tpPr(\tpX)\fto\tpPr(\tpY)$ in a straightforward way by $f_*\mu(V) = \mu(f^{-1}(V))$
The main result on Markov processes is that given the probability distribution $\mu_0$ of the state $x_0$ at time $0$, we can compute the joint probability distributions up to time $n$.
\begin{theorem}
Let $F:\tpX\fto\tpPr(\tpX)$ be a Markov process.
Then given a probability distribution $\mu_0$ of the initial state $x_0$, the probability distributions $\mu_n$ of the state $x_n$ at time $n$, and the joint probability distribution $\gamma_{n}$ of the states $(x_0,\ldots,x_n)$ up to time $n$, are computable.
\end{theorem}
\noindent
The proof is trivial given the categorical constructions of Proposition~\ref{prop:canonicaloperators}:
\begin{proof}
Compute $\mu_{n} \in \tpPr(\tpX)$ recursively by $\mu_{n}=F_*\mu_{n-1}$, which are computable by Proposition~\ref{prop:pushforward}.
Compute the joint distributions $\gamma_{n} \in \tpPr(\tpX^{n+1})$ recursively by $\gamma_0 = \mu_0$ and $\gamma_{n}=(\id \rtimes F)_*\gamma_{n-1}$, which are computable by Proposition~\ref{prop:pushforwardjoint}.
\end{proof}
\noindent
Note that $\mu_n = (\pi_n)_*\gamma_n$, where $\pi_n:\tpX^{n+1}\fto \tpX$ is given by $\pi_n(x_0,\ldots,x_n)=x_n$; in other words, the discribution at time $n$ can be extracted from the joint distribution up to time $n$.

\vspace{\baselineskip}
\noindent
We can also consider the state as a random variable on the base probability space $\Omega$.
This approach yields a random variable $X_n$ for the state at time $n$.
\begin{definition}
A parameterised Markov process on a type $\tpX$ is defined by a conditional random variable $F:\tpX \fto \tpRV(\tpX)$ and a random variable $X_0:\tpRV(\tpX)$.
\end{definition}
Given a parameterised Markov process, we can trivially extract the distribution of $X_0$ and the conditional distribution function $\Phi:\tpX\fto\tpPr(\tpX)$ by $\Phi(x)(\psi) = \Pr( \lambda\omega.\psi(F(x,\omega)) )$.

The following result shows that a parameterised Markov process gives rise to random variables $X_0,X_1,X_2,\ldots$ over the probability space $(\Omega,P)^\omega$.
\begin{theorem}
If $F:\tpX \fto(\Omega \mfto \tpX)$ is a parameterised Markov process, and $X_{\init}:\Omega \mfto \tpX$ is a random variable giving the initial probability distribution, then we can compute the stochastic process $(X_0,X_1,X_2,\ldots,)$ as a random variable $\Omega \mfto \tpX^{\infty}$.
%Alternatively, we can consider each $X_n$ as a random variable over the countable product $(\Omega,P)^{\omega}$ with $\Omega^\omega = \Omega_0 \times \Omega_1 \times \cdots$, and $X_n$ is independent of $\Omega_m$ for $m>n$.
\end{theorem}
\begin{proof}
Let $\Omega_i$ be a copy of $\Omega$ for each $i\in\N$, and define $X_n:\Omega_0\times\Omega_1\times\cdots\times\Omega_n\times\cdots \fto \tpX$ recursively by $X_0(\omega_0,\omega_1,\omega_2,\ldots)=X_{\init}(\omega_0)$ and $X_n(\omega_0,\omega_1,\ldots,\omega_n,\ldots) = F(X_{n-1}(\omega_0,\omega_1,\ldots,\omega_{n-1},\ldots))(\omega_n)$.
Then each $X_n$ is computable by computability of random variables from conditional random variables given by Theorem~\ref{thm:conditionalrandomvariable}.
Further, $X_n$ is dependent on    $(\omega_0,\omega_1,\ldots,\omega_n)$ only.
\end{proof}

%The following lemma shows that a Markov process with continuous dependence on $\tpX$ can be represented as a random variable.
%NOT TRUE!
%\begin{lemma}
%et $F:\tpX\mapsto \tpPr(\tpY)$ be a parametrised probability distribution.
%Then there exists a probability space $(\Omega,P)$ and a continuous function $G:\subset \tpX\times \Omega\rightarrow \tpY$ such that for all $x$, $E_{F(x)}(\psi) = E_\Pr( \omega \mapsto G(x,\omega) )$.
%\end{lemma}

\section{The Wiener process}

The Wiener process $W(t)$ or $W_t$ is a random process such that $W(0)=0$, the distribution function $t\mapsto W(t)$ is almost surely continuous in the weak topology, and $W(t)$ has independent increments with $W(t)-W(s) \sim N(0,t-s)$ for $0 \leq s < t$, where $N(\mu, \sigma^2)$ is the normal distribution with mean $\mu$ and variance $\sigma^2$.
The Wiener process is used in the definition of a stochastic differential equation
\[ dX(t) = f(X(t),t)\, dt + g(X(t),t)\, dW(t) . \]
There are many comprehansive books available for continuous-time stochastic processes, notably~\cite{Friedman1975,Evans2013}
%where $W(t)$ is a Wiener process (Brownian motion).
%The integral form of a stochastic differential equation is
%\[ X(t+s) - X(t) = \int_t^{t+s} f(X(\tau),\tau) \,d\tau + \int_t^{t+s} g(X(\tau),\tau)\, dW(\tau) . \]

\begin{theorem}
A sample path of the Wiener process is almost-surely $\alpha$-\Holder\ continuous for all $\alpha<1/2$.
\end{theorem}
The following result on the maximum of the Wiener process up to a given time is based on the \emph{\Andre\ reflection principle}.
\begin{theorem}
Denote by $M (t)$ the maximum of the Wiener process up to time $t$.
Then
\[ \Pr(M(t) \geq X)  = 2\Pr(W(t) > X) \,. \]
\end{theorem}

There are two main constructions of a Wiener process.
The \emph{Paley-Wiener construction} yields a Wiener process on $[0,1]$ as
\[ W(t) = A_0 t + \frac{\sqrt{2}}{\pi} \sum_{n=1}^{\infty} A_n \frac{\sin(n\pi t)}{n} \]
where the $A_n$ are independent $N(0,1)$ random variables.
The simpler \emph{\Levy-Ciesielski construction} uses wavelets.
Let $h_{n,k}$ be the $(n,k)$-th \emph{Haar function}, defined for $0 \leq k <2^n$ by
\[ h_{n,k}(x) = \begin{cases} +2^{n/2} \text{ for } \frac{k}{2^n} \leq t \leq \frac{k+1/2}{2^n} , \\ -2^{n/2} \text{ for } \frac{k+1/2}{2^n} \leq t \leq \frac{k+1}{2^n} , \\ \quad 0 \ \text{ otherwise}. \end{cases} \]
Let be $s_{n,k}$ be the $(n,k)$-th \emph{Schauder function} defined by
\[ s_{n,k}(t) = \int_{0}^{t} h_{n,k}(\tau) \, d\tau \;. \]
Note that $\sup_{t\in[0,1]}s_{n,k}(t) = 2^{-n/2}$.
Let $A_{n,k}$ be a sequence of independent $N(0,1)$ random variables on a probability space $(\Omega,P)$.
Then
\[ W(t) = \sum_{n=0}^{\infty}\sum_{k=0}^{2^n-1} A_{n,k} s_{n,k}(t) \]
is a Wiener process on $[0,1]$.

It should be noted that the the sum $\sum_{n=0}^{\infty}\sum_{k=0}^{2^n-1} A_{n,k}(\omega) s_{n,k}(t)$ does not converge for all values of the random variables $A_{n,k}$.
However, if the $A_{n,k}$ have growth bounded by $\alpha^{n/2}$ where $\alpha<2$, then $\sum_{n=0}^{\infty}\sum_{k=0}^{2^n-1} A_{n,k} s_{n,k}(t)$ converges uniformly.
By the Borel-Cantelli lemma, $\Pr(A_{n,k} \geq 2^{n/2} \text{ i.o.})=0$.

However, given only finitely many values of $A_{n,k}(\omega)$, we cannot compute a uniform approximation to the sample path $W(\omega)$, or even an approximation in $L^2([0,1])$.
In other words, the function $\omega \fto \sum A_{n,k}(\omega) s_{n,k}$ is not a computable function from $\Omega$ to $C([0,1])$ or $L^2([0,1])$.
However, it \emph{is} the case that for any open subset $U$ of $C([0,1])$, the probability $P(\{\omega \mid \sum A_{n,k}(\omega) s_{n,k} \in U \})$ is computable in $\tpHl$.
Further, there is a sequence of closed compact subsets $K_n$ of $\Omega$ such that $\Pr(K_n)\fto 1$ as $n\to\infty$ and $W$ is computable on each $K_n$.

We now give a modification of the \Levy-Ciesielski construction with base space $\Omega=\{0,1\}^\omega$ for which the Wiener process is a continuous function $W:\Omega\fto C([0,1])$.
In fact, we obtain sample paths which are \Holder-continuous in $C^\alpha$ for any $\alpha<1/2$, though we shall only prove the continuous case.
\begin{theorem}[Computable Wiener process]
Let $\Omega=\{0,1\}^\omega$ and $P$ be the standard probability measure on $\Omega$.
Then there exists a computable Wiener process $W:\Omega \pfto C[0,1]$ with open full measure domain.
\end{theorem}
\begin{proof}[Sketch of proof]
The basic idea is to modify the \Levy-Ciesielski construction so that after a finite number of bits of information we can bound the size of $A_{n,k}$ for all sufficiently large $n$.

For the event described by $|A_{n,k}|<n$ whenever $n \geq m$, we have
\[ \begin{aligned} & \prod_{n=m}^{\infty} \prod_{k=1}^{2^n} \Pr(|A_{n,k}|<n) \geq 1- \sum_{n=m}^{\infty} 2^n \Pr(|A_{n,k}| \geq n)
  \geq 1-\sum_{n=m}^{\infty} 2^n \cdot 2 \cdot \frac{1}{\sqrt{2\pi}} \int_{n}^{\infty} e^{-t^2/2}\,dt \\&\qquad
  \geq 1-\sum_{n=m}^{\infty} 2^{n+1} \cdot e^{-n^2/4} \cdot \frac{1}{\sqrt{2\pi}} \int_{n}^{\infty} e^{-t^2/4}\,dt
   \geq 1-\sum_{n=m}^{\infty} 2^{n+1} \cdot 4^{-n} \cdot 4^{-1}  = 1-\frac{1}{2^m} \,.
\end{aligned} \]
whenever $m\geq 6$, since for $n\geq 6$ we have $e^{-n^2/4} < 4^{-n}$ and $\frac{1}{\sqrt{2\pi}} \int_{n}^{\infty} e^{-t^2/4} < 1/4$.
%Note that asymptotically better bounds can be obtained.

We can therefore construct numbers $\beta_{m,n}$ such that $\beta_{m,n}=n$ whenever $n>m$, $\beta_{m+1,n}\geq \beta_{m,n}$ for all $m,n$, and \[\Pr(\forall n=0,\ldots,\infty,\ \forall k=0,\ldots,2^n-1,\ |A_{n,k}|< \beta_{m,n}) = 1/2^m . \]

We now partition a full-measure open subset of $\Omega$ into sets $\Omega_m$ of measure $1/2^{m+1}$ such that every $|A_{n,k}(\omega)| < \beta_{m,n}$ but not every $|A_{n,k}(\omega)| < \beta_{m-1,n}$ whenever $\omega\in \Omega_m$.
On each $\Omega_m$ we can computably construct the corresponding values of $A_{n,k}(\omega)$.
In particular, on $\Omega_m$, every $A_{n,k}$ is bounded and $|A_{n,k}|<n$ whenever $n>m$, so
\[ \begin{aligned}  & \textstyle \bigl| W(\omega,t) - \sum_{n=0}^{m} \sum_{k=0}^{2^n-1} A_{n,k}(\omega) s_{n,k}(t) \bigr|
    = \bigl| \sum_{n=m+1}^\infty \sum_{k=0}^{2^n-1} A_{n,k}(\omega) s_{n,k}(t) \bigr| \\
   & \qquad\qquad \qquad  \textstyle  \leq \sum_{n=m+1}^{\infty} n \cdot 2^{-n/2} \to 0 \text{ as } m\to\infty . \end{aligned} \]
\end{proof}

\section{Stochastic integration}
\label{sec:stochasticintegration}

A continuous-time real-valued stochastic process defined over the interval $[0,T]$ is a random variable taking values in $\tpCts([0,T];\tpRe$.
Since the indefinite integral $\tpCts([0,T];\tpRe) \fto \tpCts([0,T];\tpRe)$ taking $\xi$ to the function $t\mapsto\int_{0}^{t}\xi(s)\,ds$ is computable, so is the integral $  t\mapsto \int_{0}^{t} X(s)\, ds$.
In stochastic integration, we aim to give a meaning to the integral
\[ \int_{0}^{t} X(s) dW(s) \]
for a process $X$ with respect to the Wiener process.

We say that a process $X(t)$ is \emph{nonanticipative} with respect to the Wiener process if $X(t)$ depends only on $X_0$ and on $W|_{[0,t]}$, the restriction of $W$ to $[0,t]$.
Formally, letting $\mathcal{F}_t$ be the topology on $\Omega$ generated by $X_0$ and $W|_{[0,t]}$, then $X|_{[0,t]}$ is a limit of $\mathcal{F}_t$-continuous functions $X_n:\Omega\fto C([0,t];\R)$.

It turns out that this integral \emph{cannot} be computed pathwise by the Stieltjes integral.
Instead, one uses the \emph{\Ito\ integral}, which is first defined for step processes, and then extended to continuous processes.
In this section, we prove that the standard construction of the \Ito\ integral effectivises.

A stochastic process $X(\cdot)$ is a \emph{step process} if there are random variables $X_i$, $i=0,\ldots,n-1$ and times $0=t_0<t_1<\cdots<t_n=T$ such that $X(t)=X_i$ for $t\in[t_i,t_{i+1})$.
We formally write $X(t) = X_i \, I[t\in[t_i,t_{i+1})]$, where $I[t\in[t_i,t_{i+1})]$ is the indicator function with value $1$ if $t\in[t_i,t_{i+1})$ and $0$ otherwise.
It is straightforward to show that if $\Ex(X_i^2)<\infty$ for all $i$, then the step process is well-defined as an element of $M^2(L^2([0,T];\R))$, where $L^2(L^2([0,T];\R)$ is the space of Lesbesgue-integrable functions on $[0,T]$, and $M^2$ the space of square-integrable random variables.

We first show that given $\xi \in C([0,T];\R)$, we can compute step functions $\eta$ taking values in the Lesbesgue space $L^2([0,T];\R)$.
\begin{theorem}\label{thm:computablestepfunction}
Given $\xi:C([0,T];\R)$, we can compute a sequence of step function $\eta_n:[0,T]\fto \R$ such that $\eta_n\to\xi$ effectively in $L^2([0,T];\R)$.
\end{theorem}
\marginpar{Do we use this? Address convergence of step *processes*}
\begin{proof}
Choose a sequence $\delta_n>0$ effectively converging to $0$, an choose partitions $\calT_n = \{0=t_{n,0}<t_{n,1}<\cdots<t_{n,m_n}=T\}$ where each $t_{n,i}$ is computable and $t_{n,i+1}-t_{n,i}<\delta_n$ for all $n,i$.
Compute $\eta_{n,i}=\xi(t_{n,i})$ and define $\eta_n(t)=\eta_{n,i}$ for $t_{n,i}\leq t<t_{n,i+1}$.
Clearly $\eta_n \in L^2([0,T];\R)$.
The integral $\int_{t=0}^{T} (\xi(t)-\eta_n(t))^2\,dt = \sum_{i=0}^{m_n-1} \int_{t_{i}}^{t_{i+1}} (\xi(t)-\xi(t_{n,i}))^2\,dt$ is computable, and converges to $0$ as $n\to\infty$ by continuity of $\xi$, so convergence is effective.
\end{proof}
\noindent
Note that continuity of $\xi$ is required to compute $\xi(t_i)$, but we do not need to know the modulus of continuity to compute the rate of convergence of $\eta_n$.
By Theorem~\ref{thm:imagerandomvariable}, this pathwise computation extends to random variables, and it is clear that if $X$ is nonanticipative with respect to $W$, then so are the step processes $X_n$.

\begin{definition}[\Ito\ integral for step processes]
Given a step process $X = \sum_{i=0}^{n-1} X_i \, I[t\in[t_i,t_{i+1})]$, we define the \emph{\Ito\ integral} as
\[ \int_{0}^{T} X(t) dW(t) = \sum_{i=0}^{n-1} X_i \bigl( W(t_{i+1})-W(t_i) \bigr) . \]
This definition can be extended to an indefinite integral:
Take $m(s) = \max\{i\mid t_i<s\}$ and define
\begin{equation}
\label{eq:indefiniteitointegral}
\int_{0}^{t} X(t) \,dW(t) = \sum_{i=0}^{m(s)-1}  X_i \bigl(W(t_{i+1})-W(t_i)\bigr) + X_{m(s)} \bigl(W(s)-W(t_{m(s)})\bigr)  .
\end{equation}
\end{definition}

\begin{lemma}
The \Ito\ integral of a step process is computable as a continuous process.
Further, if $X(\cdot)$ is nonanticipative with respect to the Wiener process, then so is its \Ito\ integral.
\end{lemma}
\begin{proof}
By the defining equation~\eqref{eq:indefiniteitointegral}, $W(t)$ is continuous when restricted to each interval $[t_i,t_{i+1}]$, and clearly the integral is continuous over the step boundaries.
It is also clear that the \Ito\ integral is nonanticipative, at time $t$ since it depends only on $W(s)$ for $s\leq t$.
\end{proof}

The following \emph{\Ito~equality} is crucial, since it relates the stochastic integral with an ordinary integral.
\begin{lemma}
If $X=\sum X_k \chi_{[t_k,t_{k+1}})$ is a step process, and $X_k$ is independent of $W(t)-W(s)$ for all $t>s>t_k$, then
\begin{equation} \label{eq:stepitoisometry} \Ex\Bigl( \int_{0}^{T} X(t)\,dW(t) \Bigr)^{\!2} = \Ex \int_{0}^{T} X(t)^2 \, dt . \end{equation}
\end{lemma}
\begin{proof}
\[ \begin{aligned}
  \Ex\Bigl( \int_{0}^{T} X(t)\,dW(t) \Bigr)^2 &= \Ex\Bigl( \int_{0}^{T} X(s)\,dW(s)\int_{0}^{T} X(t)\,dW(t)\Bigr) \\
    &= \Ex\Bigl(\sum_{i=0}^{m-1}  X_i \bigl(W(t_{i+1})-W(t_i)\bigr) \sum_{j=0}^{m-1}  X_j \bigl(W(t_{j+1})-W(t_j)\bigr) \Bigr) \\
    &= \sum_{i,j=0}^{m-1} \Ex\bigl(X_i X_j (W(t_{i+1})-W(t_i)) (W(t_{j+1})-W(t_j)) \bigr)
   \end{aligned} \]
If $i<j$, then since $X_i$, $X_j$ are independent of $(W(t_{j+1})-W(t_j))$,
\[ \begin{aligned} & \Ex\bigl( X_i X_j (W(t_{i+1})-W(t_i)) (W(t_{j+1})-W(t_j)) \bigr)
   \\[\jot] &\qquad = \Ex\bigl(X_i X_j (W(t_{i+1})-W(t_i))\bigr) \; \Ex\bigl(W(t_{j+1})-W(t_j)\bigr) = 0 \end{aligned}  \]
and a similar estimate holds for $i>j$.
Hence
\[ \begin{aligned}
  \Ex\Bigl( \int_{0}^{T} X(t)\,dW(t) \Bigr)^{\!2}
    &= \sum_{i=0}^{m-1} \Ex\bigl( X_i^2 (W(t_{i+1})-W(t_i))^2 \bigr)
     = \sum_{i=0}^{m-1} \Ex(X_i^2) \Ex\bigl(W(t_{i+1})-W(t_i)\bigr)^2 \\
    &= \sum_{i=0}^{m-1} \Ex(X_i^2) (t_{i+1}-t_i)
     = \Ex\Bigl(\sum_{i=0}^{m-1} X_i^2(t_{i+1}-t_i) \Bigr) \\
    &= \Ex\Bigl(\int_{0}^{T} X(t)^2\,dt \Bigr) .
 \end{aligned} \vspace{-\baselineskip} \]
 \qedhere
\end{proof}

In order to bound the expected maximum value along a path, we will use \emph{martingale} properties of the integrated process.
\begin{definition}[(Sub)martingale]
A discrete stochastic process $X$ is a \emph{martingale} if for all $k$, $\Ex(|X_k|)<\infty$ and $\Ex(X_{k+1}|\mathcal{F}_k)=X_k$, and a \emph{submartingale} if  for all $k$, $\Ex(|X_k|)<\infty$ and $\Ex(X_{k+1}|\mathcal{F}_k)\geq X_k$.

A stochastic process $X$ is a \emph{martingale} if for all $t$, $\Ex(|X_t|)<\infty$ and for all $t>s$, $\Ex(X_{t}|X_{s})=X_s$, and a \emph{submartingale} if $\Ex(X_{t}|X_{s}) \geq X_{s}$.
\end{definition}
\noindent
We can give sufficient conditions for a process to be a (sub)martingale avoiding the use of conditional expectation.
We say $X$ has \emph{independent increments} if for any $t_0<t_1<\cdots<t_n$, the increments $X_{t_{i}}-X_{t_{i-1}}$ for $i=1,\ldots,n$ are all independent.
If $X$ has independent increments, then $X$ is a martingale if $\Ex(X_t-E_s)=0$ whenever $t>s$ and a submartingale if $\Ex(X_t-X_s) \geq 0$.

\begin{lemma}
The \Ito\ integral of a step process has independent increments with zero expectation, so is a martingale.
\end{lemma}
\begin{proof}
Let $Y$ be the integrated process $Y(t) = \int_{0}^{t} X(s) dW(s) =  Y(s) + \int_{s}^{t} X(r) \,dW(r)$. 
Then
\begin{multline*} \textstyle  Y(t)  =  Y(s) + X_{m(s)} \bigr( W(t_{m(s)+1})-W(s)\bigr)  + \sum_{i=m(s)}^{m(t)-1}  X_i \bigl(W(t_{i+1})-W(t_i)\bigr) Y(s) \\\textstyle +  X_{m(t)} \bigl(W(t)-W(t_{m(t)})\bigr) . \end{multline*}
Since $W(t_3)-W(t_2)$ is independent of $X(t_1)$ whenever $t_1<t_2<t_3$, we have $\Ex(Y(t)) = \Ex(Y(s))$ or $\Ex(Y(t)-Y(s))=0$.
\end{proof}

\begin{lemma}
\label{lem:martingalejensen}
If $(X_k)_{k=1,\ldots,m}$ is a martingale and $\phi$ is convex, then $\bigl(\phi(X_k)\bigr)_{k=1,\ldots,m}$ is a submartingale.
Similarly, if $X(\cdot)$ is a martingale, then $\phi(X(\cdot))$ is a submartingale.
\end{lemma}
\begin{proof}
By Jensen's inequality, $\Ex(\phi(X_{k+1})|\mathcal{F}_k) \geq  \phi(\Ex(X_{k+1}|\mathcal{F}_k)) = \phi(\Ex(X_k))$.
\end{proof}

We now give some estimates known as \emph{martingale inequalities} which will allow us to compute limits of processes in $C([0,T];\R)$.

\begin{lemma}[Discrete submartingale inequality]
\label{lem:discretesubmartingale}
Let $(X_k)$ be a discrete positive submartingale, and $Y_n=\max_{k=1,\ldots,n} X_k$.
Then for any $\lambda>0$,
\begin{equation} \label{eq:discretesubmartingaleinequality}
 \lambda \Pr\bigl( Y_n \geq \lambda \bigr) \leq \Ex\bigl(X_n \, I[{Y_n\geq\lambda}] \bigr) \leq \Ex X_n . \end{equation}
\end{lemma}
\noindent
Note that the expression $\Ex\bigl(X_n \, I[{Y_n\geq\lambda}])$ makes sense since the function $\max:\R^n\fto \R$ is computable and the characteristic function $\xi_{[\lambda,\infty)}$ is computable $\R\fto[0,1]_>$ for a given $\lambda$.
Hence $\Ex\bigl(X_n \, I[{Y_n\geq\lambda}]\bigr)$ is computable as the upper integral $\Ex_>(f(X_1,\ldots,X_n))$ where $f=\chi_{\geq\lambda}\circ\max$.

The basic idea of the proof is to consider events $X_k\geq\lambda$ and for all $i<k$, $X_i<\lambda$.
However, these sets of events are neither closed nor open, so we cannot directly probabilities or expectations.
\begin{proof}
Let $A$ be the event $\max_{k=1,\ldots,n}X_k\geq\lambda$.
For fixed $\delta>0$, let $A_{\delta,k}$ be the event $\bigwedge_{i=1}^{k-1} (X_i \leq \lambda-\delta) \wedge (X_k \geq \lambda)$, and $A_\delta$ the event $\bigcup_{k=1}^{n}A_{\delta,k}$.
Note that $A=\bigcup_{\delta > 0}A_\delta$.
Since $A_{\delta,k}$ holds on a closed set, the characteristic function is upper semicontinuous and we can consider the upper horizontal integral of $\chi_{A_{\delta,k}}$
Since $X_n-X_k$ is independent of $X_1,\ldots,X_k$ for $k<n$, and hence independent of $A_{\delta,k}$, we have
\[ \begin{aligned} & \textstyle \Ex(X_n \chi_{A_{\delta,k}}) = \Ex((X_n-X_k)\chi_{A_{\delta,k}})+\Ex(X_k\chi_{A_{\delta,k}}) \\ & \qquad \qquad \textstyle =\Ex(X_n-X_k)\Ex(\chi_{A_{\delta,k}})+\Ex(X_k\chi_{A_{\delta,k}})\geq \Ex(X_k\chi_{A_{\delta,k}}) . \end{aligned}  \]
Then summing probabilities over each $A_{\delta,k}$ gives
\[ \begin{aligned} & \textstyle \lambda \Pr(X \in A_{\delta}) \leq \sum_{k=1}^{n} \lambda \Pr(A_{\delta,k}) = \sum_{k=1}^{n} \Ex(\lambda \chi_{A_{\delta,k}}) \\ & \qquad \textstyle  \leq \sum_{k=1}^{n} \Ex(X_k \chi_{A_{\delta,k}}) \leq \sum_{k=1}^{n} \Ex(X_n\chi_{A_{\delta,k}}) ) = \Ex\bigl( \sum_{k=1}^{n}X_n\chi_{A_{\delta,k}}\bigr) = \Ex\bigl( X_n \chi_{A_{\delta}} ) \end{aligned} \]
By Lemma~\ref{lem:decreasingopen}, $\Pr(X\in A_\delta)\to\Pr(X\in A)$ as $\delta\to 0$, and the result follows.
\end{proof}
We can also show the standard result that if $X_t$ is a martingale, and $\Ex|X_T|^\alpha<\infty$ for some $\alpha\geq1$, then
\( \Pr\bigl( {\textstyle\max_{t\leq T}} |X_t|  \geq \epsilon \bigr) \leq \frac{1}{\epsilon^\alpha} \Ex|X_T|^\alpha , \)
but will not need this in the sequel.
Instead, we use the following extension to square-integrable martingales:
\begin{lemma}[Discrete integrable martingale inequality]%; Friedman Theorem~4.3.7]
\label{lem:discreteintegrablemartingaleinequality}
Let $(X_k)_{k=1,\ldots,n}$ be a discrete martingale.
%Then for any $\alpha>1$, \begin{equation} \label{eq:continuousintegrablemartingaleinequality}  \Ex\bigl( {\textstyle\max_{t\in[0,T]}} |X(t)|^\alpha\bigr) \leq \Bigl(\frac{\alpha}{1-\alpha}\Bigr)^{\!\alpha} \Ex|X(T)|^\alpha . \end{equation}
Then
\begin{equation}  \label{eq:discreteintegrablemartingaleinequality}  \Ex\bigl( {\textstyle\max_{k=1,\ldots,n}} |X_k|^2\bigr) \leq 4 \Ex(X_n^2) . \end{equation}
\end{lemma}
\noindent
\begin{proof}
Define $Y=\max_{k=1,\ldots,n} |X_k|$ and $Z=X_n$.
By the stronger form of the submartingale inequality of Lemma~\ref{lem:discretesubmartingale}, we have $\lambda \Pr(Y \geq \lambda) \leq \Ex( I[Y \geq \lambda] Z )$.
Since $Y=\int_{0}^{\infty} I[Y\geq\lambda] \, d\lambda$ and $ Y^2 = 2 \int_{0}^{\infty} \lambda \, I[Y\geq\lambda] \, d\lambda$, we have
\[  \Ex Y^2 = 2 \Ex \int_{0}^{\infty} \lambda \, I[Y\geq\lambda] \, d\lambda \leq 2 \Ex \int_{0}^{\infty} I[Y\geq \lambda] \, Z d\lambda = 2 \Ex \Bigl( Z \int_{0}^\infty I[Y\geq\lambda] d\lambda \Bigr) = 2\Ex(Z Y) . \]
\Holder's inequality gives
\( \Ex(YZ) \leq  \bigl(\Ex(Y^2)\bigr)^{1/2} \bigl(\Ex(Z^2)\bigr)^{1/2} . \)
Therefore we have
\( \Ex(Y^2) \leq 2 \Ex(YZ) \leq 2 \bigl(\Ex(Y^2)\bigr)^{1/2} \bigl(\Ex(Z^2)\bigr)^{1/2}. \)
Hence \( (\Ex Y^2)^{1/2} \leq 2 (\Ex Z^2 )^{1/2}  \) yielding $\Ex Y^2 \leq 4\Ex Z^2$.
\end{proof}

The results above on discrete (sub)martingales extend to continuous processes.
We will require:
\begin{theorem}[Integrable martingale inequality]%; Friedman Theorem~4.3.7]
Let $X(\cdot)$ be a martingale.
%Then for any $\alpha>1$, \begin{equation} \label{eq:continuousintegrablemartingaleinequality}  \Ex\bigl( {\textstyle\max_{t\in[0,T]}} |X(t)|^\alpha\bigr) \leq \Bigl(\frac{\alpha}{1-\alpha}\Bigr)^{\!\alpha} \Ex|X(T)|^\alpha . \end{equation}
Then
\begin{equation}  \label{eq:continuousintegrablemartingaleinequality}  \Ex\bigl( {\textstyle\max_{t\in[0,T]}} |X(t)|^2\bigr) \leq 4 \Ex(X(T)^2) . \end{equation}
\end{theorem}
\noindent
\begin{proof}
Let $\{t_1,t_2,\ldots\}$ be a dense subset of $[0,T]$.
Then by Lemma~\ref{lem:discreteintegrablemartingaleinequality}, $\Ex\bigl(\max_{t\in[0,T]} |X(t)|^2\bigr) = \lim_{n\to\infty} \Ex\bigl(\max_{k=1,\ldots,n}|X(t_{k})|^2\bigr) \leq \lim_{n\to\infty} 4\Ex\bigl(X(T)^2\bigr) = 4\Ex\bigl(X(T)^2\bigr)$.
\end{proof}
\omittext{
\begin{proof}
Define $Y=\max_{t\in[0,T]} |X(t)|$ and $Z=X(T)$.
By the stronger form of the submartingale inequality of Lemma~\ref{lem:continuoussubmartingale}, we have $\lambda \Pr(Y \geq \lambda) \leq \Ex( I[Y \geq \lambda] Z )$
Since $ Y^2 = 2 \int_{0}^{\infty} \lambda \, I[Y>\lambda] \, d\lambda$, we have
\[  \Ex Y^2 = 2 \Ex \int_{0}^{\infty} \lambda \, I[Y>\lambda] \, d\lambda \leq 2 \Ex \int_{0}^{\infty} I[Y\geq \lambda] \, Z d\lambda = 2 \Ex \Bigl( Z \int_{0}^\infty I[Y\geq\lambda] d\lambda \Bigr) = 2\Ex(Z Y) . \]
\Holder's inequality gives
\( \Ex(YZ) \leq  \bigl(\Ex(Y^2)\bigr)^{1/2} \bigl(\Ex(Z^2)\bigr)^{1/2} . \)
Therefore we have
\( \Ex(Y^2) \leq 2 \Ex(YZ) \leq 2 \bigl(\Ex(Y^2)\bigr)^{1/2} \bigl(\Ex(Z^2)\bigr)^{1/2}. \)
Hence \( (\Ex Y^2)^{1/2} \leq 2 (\Ex Z^2 )^{1/2}  \) yielding $\Ex Y^2 \leq 4\Ex Z^2$.
\end{proof}
}%\omit

Combining these results, we see that for nonanticipative step processes $X(\cdot)$, the \Ito\ integral $\int_{s=0}^{t} X(s)dW(s)$ is computable and is a continuous martingale.
The \Ito\ equality
\begin{equation*}  \Ex\bigl( {\textstyle\int_{0}^{T} X(t)dW(t)} \bigr)^2 = \Ex\bigl({\textstyle\int_{0}^{T} X(t)dt} \bigr)^2 . \end{equation*}
shows that the integrated process is a square-integrable random variable, and the martingale inequality
\begin{equation*}   \Ex\bigl( {\textstyle\max_{t\in[0,T]}\bigl|\int_{0}^{t} X(s)dW(s)\bigr|} \bigr)^2 \leq 4 \Ex\bigl( {\textstyle\int_{0}^{T} X(t)dW(t)} \bigr)^2 \end{equation*}
gves uniform bounds on the integrated process.
Combining these inequalities gives
\begin{equation*}  \Ex\bigl( {\textstyle\max_{t\in[0,T]}\bigl|\int_{0}^{t} X(s)dW(s)\bigr|} \bigr)^2 \leq 4 \Ex\bigl({\textstyle\int_{0}^{T} X(t)dt} \bigr)^2 . \end{equation*}
In terms of norms on the processes, we have
\begin{equation*} \textstyle ||\!\int\!X dW||_{\infty,2} \leq 2 || X ||_{2,2} . \end{equation*}
Now if $X_n$ is a Cauchy sequence of step processes converging effectively to $X_\infty$ in the $||\cdot||_{2,2}$ norm, we have form $m\geq n$,
\begin{equation*} \textstyle ||\!\int\!X_m\!-\!X_n\,dW||_{\infty,2} \leq 2 || X_m-X_n ||_{2,2} \leq 2 \bigl( || X_m-X_\infty ||_{2,2} + || X_n-X_\infty ||_{2,2} \bigr). \end{equation*}
Hence $\int\!X_ndW$ converges effectively in the $||\cdot||_{\infty,2}$ norm.
Further, for continuous processes
\begin{equation*}
\textstyle \int_{0}^{T}X(t)^2\,dt \leq T \,\max_{t\in[0,T]}(X(t)^2) = T(\max_{t\in[0,T]}|X(t)|)^2 ,
\end{equation*}

\begin{theorem}[Computability of the \Ito\ integral]
\label{thm:itocomputability}
If $X$ is a square-integrable step process, then $\int\!XdW$ a continuous process computable from $X$.
If $(X_n)_{n\in\N}$ a sequence of step processes converging effectively to $X_\infty$ in the $||\cdot||_{2,2}$ norm, then $\int\!X_ndW$ is a Cauchy sequence converging effectively in the $||\cdot||_{\infty,2}$ norm to a process which we define as $\int\!X_\infty\,dW$.
Further, if $X$ is a continuous process, then
\[ \textstyle  ||\int\!XdW||_{\infty,2} \leq 2\sqrt{T} \, ||X||_{\infty,2} . \]
\end{theorem}

\marginpar{From referee's comments}
\begin{remark}
If $X$ is a continuous computable non-anticipative process such that $||X||_{\infty,2}<\infty$, then $\int X dW$ is computable in $\tpRV(\tpCts([0,1]\fto\tpR))$ from process $X$.
\end{remark}
\begin{remark}
Also, Theorem~\ref{thm:itocomputability}, which asserts the computability of the \Ito\ integral, is
much more general than the author seems to imply.
First, there are some common \Ito\o integrable processes, such as $\sign W(t)$, which are neither step processes nor have continuous (even right-continuous) paths. Moreover, it is implicit in the author's work how to handle these processes computably.
The space of square $W$-integrable processes (up to some basic equivalence) is just the closed subspace of $L^2(\Omega \times [0,T], P \otimes dt)$ spanned by the nonanticipative step processes relative to $W$.
Then it is clear that for any $X$ in this space, $X$ is a function of $W$ , and $X_t(\omega) = X(W(\omega),t)$ is a random variable (computable from $W$ and $X$) where $(X,W)$ has the desired joint distribution with $W$.
\end{remark}
\marginpar{Make explicit the data type of Ito integrable processes.}

\section{Stochastic differential equations}

We now consider stochastic differential equations
\[ dX(t) = f(X(t))\,dt + g(X(t))\,dW(t); \qquad X(0)=X_0 . \]
The integral form is
\[ X(t) = X_0 + \int_{0}^{t} f(X(s))\,ds + \int_{0}^{t} g(X(s))\,dW(s) =: J[X](t). \]
We assume $f,g:\R\fto\R$ are Lipschitzian functions with constants $K$ and $L$, respectively.
We first assume $X_0 \in M^2(\R)$ is square-integrable, though the case of a constant $X_0=x_0$ will actually suffice.

\begin{lemma}
Suppose $f:\R\fto\R$ is a Lipschitzian function with constant $K$, and suppose $Y,Z$ are stochastic processes with $Y(0)=Z(0)$.
Then
\begin{equation} d_{2,\infty}(f(Y),f(Z)) \leq K d_{2,\infty}(Y,Z) . \end{equation}
\end{lemma}
\begin{proof}
\( \Ex\bigl( {\textstyle\max_{t\in[0,T]}} |f(Y(t))-f(Z(t))| \bigr)^2 \leq \Ex\bigl(  {\textstyle\max_{t\in[0,T]}} K|Y(s)-Z(s)|\bigr)^2 \leq K^2 \Ex\bigl(  {\textstyle\max_{t\in[0,T]}} |Y(s)-Z(s)|\bigr)^2 . \)
\end{proof}
\begin{lemma}[Computability of non-stochastic integrals]
The non-stochastic integral $t\mapsto \int_{0}^{t} X(s)\,ds$ of a continuous stochastic process is computable. Further,
\begin{equation} d_{2,\infty}\bigl({\textstyle \int Y dt},{\textstyle\int Z dt}\bigr) \leq T d_{2,\infty}(Y,Z) . \end{equation}
\end{lemma}
\begin{proof}
Computability is immediate from computability of non-random integrals.
Further, we have
\begin{align*} \Ex\Bigl( \max_{t\in[0,T]} \Bigl| \int_{0}^{t} X(s)\,ds \Bigr| \Bigr)^2
    &\leq \Ex\Bigl( \max_{t\in[0,T]}   \int_{0}^{t} |X(s)|\,ds \Bigr)^2
     \leq \Ex\Bigl(\int_{0}^{T} |X(t)|\,dt \Bigr)^{\!2} \\
    &\leq \Ex\bigl( T\,{\textstyle\max_{t\in[0,T]}}|X(t)|\bigr)^2
     =    T^2\, \Ex\bigl({\textstyle\max_{t\in[0,T]}} |X(t)|\bigr)^2 . \qedhere
\end{align*}
\end{proof}
%\[ d_{2,\infty}\bigl({\textstyle t\mapsto\int_{0}^{t}Y(s)ds},{\textstyle t\mapsto\int_{0}^{t}Z(s)ds}\bigr) \leq T d_{2,\infty}(Y,Z) . \]
For stochastic integrals, we recall from Section~\ref{sec:stochasticintegration} that
\begin{align*}
  \Ex\Bigl( \max_{t\in[0,T]} \Bigl| \int_{0}^{t} X(s)\,dW(s) \Bigr| \Bigr)^2
%    &\leq  4\,\Ex\!\int_{0}^{T} X(s)^2\,ds
%     \leq  4\,\Ex\!\int_{0}^{T} \bigl({\textstyle\max_{r\in[0,T]}} X(r)\bigr)^2\,ds \\
    &\leq  4T\,\Ex\bigl( {\textstyle\max_{r\in[0,T]}} |X(r)|\bigr)^2\,ds %\leq 4 \Ex \int_{0}^{T} X(s)^2\,ds \sup_{t\in[0,T]}
\end{align*}
so
%\[ d_{2,\infty}\bigl({\textstyle t\mapsto\int_{0}^{t}Y(s)dW(s)},{\textstyle t\mapsto\int_{0}^{t}Z(s)dW(s)}\bigr) \leq 2 \sqrt{T} d_{2,\infty}(Y,Z) . \]
\[ d_{2,\infty}\bigl({\textstyle \int\! Y dW},{\textstyle\int\! Z dW}\bigr) \leq 2 \sqrt{T} d_{2,\infty}(Y,Z) . \]
Now
\begin{equation*}
  J[Y](t)-J[Z](t)  =  \int_{0}^{t} f(Y(s))-f(Z(s)) \,ds + \int_{0}^{t} g(Y(s)) - g(Z(s))\,dW(s)
 %\Bigr( X_0 + \int_{0}^{t} f(Y(s))\,ds + \int_{0}^{t} g(Y(s))\,dW(s) \Bigr) - \Bigr( X_0 + \int_{0}^{t} f(Z(s))\,ds + \int_{0}^{t} g(Z(s))\,dW(s) \Bigr) \\
\end{equation*}
The expected square of the supremum of this quantity is
\[ \begin{aligned}
 & \textstyle d_{2,\infty}(J[Y],J[Z]) = ||J[Y-Z]||_{2,\infty} \leq ||\int\! f(Y)-f(Z)\,dt||_{2,\infty} + ||\int\! g(Y)-g(Z)\,dW ||_{2,\infty} \\
 & \qquad \textstyle = d_{2,\infty}\bigl(\int\!f(Y)dt,\int\!f(Z)dt\bigr)+d_{2,\infty}\bigl(\int\!g(Y) dW,\int\!g(Z) dW\bigr) \\
  & \qquad\qquad  \textstyle \leq T d_{2,\infty}(f(Y),f(Z))+2\sqrt{T} d_{2,\infty}(g(Y),g(Z)) \leq (TK+2\sqrt{T}L)d_{2,\infty}(Y,Z) .
\end{aligned}\]
Hence we have
\[ d_{2,\infty}(J[Y]-J[Z]) \leq (KT+2L\sqrt{T}) \, d_{2,\infty}(Y,Z) .\]
Taking $T<\min(1/2K,1/16L^2)$ gives $(KT+2L\sqrt{T})<1$, so
\[ d_{2,\infty}(J[Y],J[Z]) \leq \kappa\,d_{2,\infty}(Y,Z) \]
where $\kappa<1$.

The integral form of the equation i.e. the Picard operator, is therefore a contraction map for small enough $T$.
Taking $X_0$ to be the constant process, $X_0(t) \equiv X_0$, and setting $X_{n+1} = J[X_{n}]$ for all $n\in\N$ gives an effectively convergent subsequence.
The initial difference is given by
\[ d_{2,\infty}(X_1-X_0)  \leq \kappa \, \bigl(\Ex X_0^2\bigr)^{1/2}  .\]

By joining together computations of the solution for small enough $T$, we obtain:
\begin{theorem}[Computability of Lipschitz stochastic differential equations]
Consider the stochastic differential equation
\[ dX(t) = f(X(t))\, dt + g(X(t)) \, dW(t); \qquad X(0)=X_0 \]
where $f,g:\R\fto\R$ are Lipschitz and $X_0\in R(\R)$.
Then the solution $X(t)$ is computable as a random variable taking values in $C([0,\infty);\R)$.
\end{theorem}
\begin{proof}
Let $K$ be the Lipschitz constant for $f,g$, and choose $T<\min(4,1/16K^2)$.
For a given $x_0$, the solution with initial condition $X(0)=x_0$ is computable in $M^2(C([0,T];\R)$, and hence in $R(C([0,T];\R)$.
Hence the solution operator given initial condition $x_0\in\R$ is computable $\R\fto R(C([0,T];\R)$.
For an initial condition which is a probability distribution over $\R$, the solution is computable over $[0,T]$ by Theorem~\ref{thm:itocomputability}.
Then the random variable $X(T)$ is computable by projection.
The result follows by recursively computing $X$ over the intervals $[kT,(k+1)T]$.

\end{proof}

\section{Conclusions}

In this paper, we have developed a theory of probability, random variables and stochastic processes which is sufficiently powerful to effectively compute the solution of stochastic differential equations.
\marginpar{Can we handle conditioning on sub-sigma-algebras?}
The theory uses type-two effectivity to provide an underlying machine model of computation, but is largely developed using type theory in the cartesian-closed category of quotients of countably-based spaces, which has an effective interpretation.
The approach extends existing work on probability via valuations and random variables in metric spaces via limits of Cauchy sequences.
Ultimately, we hope that this work will form a basic for practical software tools for the rigorous computational analysis of stochastic systems.
\vspace{\baselineskip}

\noindent
\textbf{Acknowledgement:} The author would like to thank Bas Spitters for many interesting discussions on measurable functions and type theory, and pointing out the connection with monads.

\bibliographystyle{alpha}
\bibliography{computablestochastic}

\begin{thebibliography}{HRW11}

\bibitem[AM02]{AlvarezManilla2002}
Mauricio Alvarez-Manilla.
\newblock Extension of valuations on locally compact sober spaces.
\newblock {\em Topology Appl.}, 124:397--433, 2002.

\bibitem[BB85]{BishopBridges1985}
Errett Bishop and Douglas Bridges.
\newblock {\em Constructive analysis}, volume 279 of {\em Grundlehren der
  Mathematischen Wissenschaften}.
\newblock Springer, 1985.

\bibitem[BC72]{BishopCheng1972}
Errett Bishop and Henry Cheng.
\newblock {\em Constructive measure theory}.
\newblock American Mathematical Society, 1972.

\bibitem[Bra01]{Brattka2001}
Vasco Brattka.
\newblock Computable versions of {Baire}'s category theorem.
\newblock In {\em Proc. 26th International Symposium on Mathematical
  Foundations of Computer Science}, pages 224--235. Springer, 2001.

\bibitem[Bra05]{Brattka2005}
Vasco Brattka.
\newblock Effective {Borel} measurability and reducibility of functions.
\newblock {\em Math. Logic Quarterly}, 51:19--44, 2005.

\bibitem[Cha72]{Chan1972}
Yuen-Kwok Chan.
\newblock A constructive approach to the theory of stochastic processes.
\newblock {\em Transactions of the American Mathematical Society}, 165:37--44,
  1972.

\bibitem[Col09]{Collins2009CIE}
Pieter Collins.
\newblock Computable types for dynamic systems.
\newblock In {\em Proceedings of the Fifth Conference on Computability in
  Europe}, 2009.

\bibitem[CS09]{CoquandSpitters2009}
Thierry Coquand and Bas Spitters.
\newblock Integrals and valuations.
\newblock {\em J. Logic Analysis}, 1(3):1--22, 2009.

\bibitem[Dob07]{Doberkat2007}
Ernst-Erich Doberkat.
\newblock {\em Stochastic Relations: Foundations for {M}arkov Transition
  Systems}.
\newblock Chapman \& Hall, 2007.

\bibitem[Eda95a]{Edalat1995DTI}
Abbas Edalat.
\newblock Domain theory and integration.
\newblock {\em Theor. Comput. Sci.}, 151:163--193, November 1995.

\bibitem[Eda95b]{Edalat1995DSM}
Abbas Edalat.
\newblock Dynamical systems, measures, and fractals via domain theory.
\newblock {\em Inf. Comput.}, 120:32--48, July 1995.

\bibitem[Esc04]{Escardo2004}
Mart\'{\i}n Escard\'{o}.
\newblock Synthetic topology.
\newblock {\em Electron. Notes Theor. Comput. Sci.}, 87:21--156, November 2004.

\bibitem[Esc09]{Escardo2009}
Mart\'{\i}n Escard\'{o}.
\newblock Semi-decidability of may, must and probabilistic testing in a
  higher-type setting.
\newblock {\em Electron. Notes Theor. Comput. Sci.}, 249:219--242, August 2009.

\bibitem[Eva13]{Evans2013}
Craig Evans.
\newblock {\em An Introduction to Stochastic Differential Equations}.
\newblock AMS, 2013.

\bibitem[Fri75]{Friedman1975}
Avner Friedman.
\newblock {\em Stochastic Differential Equations}.
\newblock Dover, 1975.

\bibitem[GLV11]{GoubaultLarrecqVaracca2011}
Jean Goubault-Larrecq and Daniele Varacca.
\newblock Continuous random variables.
\newblock In {\em Proceedings of the 2011 IEEE 26th Annual Symposium on Logic
  in Computer Science}, pages 97--106, Washington, DC, USA, 2011.

\bibitem[GS04]{GikhmanSkorohod2004}
I.~Gikhman and A.~Skorohod.
\newblock {\em The Theory of Stochastic Processes 1}.
\newblock Springer, 2004.

\bibitem[HR09]{HoyrupRojas2009IC}
Mathieu Hoyrup and Crist\'obal Rojas.
\newblock Computability of probability measures and {Martin}-{L\"of} randomness
  over metric spaces.
\newblock {\em Information and Computation}, 207:830--847, 2009.

\bibitem[HRW11]{HoyrupRojasWeihrauch2011}
Mathieu Hoyrup, {Crist\'obal} Rojas, and Klaus Weihrauch.
\newblock Computability of the radon-nikodym derivative.
\newblock In Benedikt {L\"owe}, Dag Normann, Ivan Soskov, and Alexandra
  Soskova, editors, {\em Models of Computation in Context}, volume 6735 of {\em
  Lecture Notes in Computer Science}, pages 132--141. Springer, 2011.

\bibitem[JP89]{JonesPlotkin1989}
C.~Jones and G.~Plotkin.
\newblock A probabilistic powerdomain of evaluations.
\newblock In {\em Proceedings of the Fourth Annual Symposium on Logic in
  computer science}, pages 186--195, Piscataway, NJ, USA, 1989.

\bibitem[Ker08]{Kersting2008}
G{\"o}tz Kersting.
\newblock Random vaiables --- without basic space.
\newblock In J.~Blath, P.~{M\"orters}, and M.~Scheutzow, editors, {\em Trends
  in Stochastic Analysis}. Cambridge University Press, 2008.

\bibitem[K{\"o}n97]{Konig1995}
H.~K{\"o}nig.
\newblock {\em Measure and Integration}.
\newblock Springer-Verlag, 1997.

\bibitem[Law04]{Lawson2004}
Jimmie~D. Lawson.
\newblock Domains, integration and `positive analysis'.
\newblock {\em Mathematical. Structures in Comp. Sci.}, 14:815--832, December
  2004.

\bibitem[Mis07]{Mislove2007}
Michael Mislove.
\newblock Discrete random variables over domains.
\newblock {\em Theor. Comput. Sci.}, 380:181--198, July 2007.

\bibitem[MW43]{MannWald1943}
H.B. Mann and A.~Wald.
\newblock On stochastic limit and order relationships.
\newblock {\em Ann. Math. Statistics}, 14(3):217--226, 1943.

\bibitem[OS10]{OConnorSpitters2010}
Russell O'Connor and Bas Spitters.
\newblock A computer-verified monadic functional implementation of the
  integral.
\newblock {\em Theor. Comput. Sci.}, 411(37):3386--3402, 2010.

\bibitem[Roh52]{Rokhlin1952}
V.~A. Rohlin.
\newblock {\em On the fundamental ideas of measure theory}, volume~71 of {\em
  Translations}.
\newblock American Mathematical Society, 1952.
\newblock Translated from Russian.

\bibitem[Sch07]{Schroder2007}
Matthias Schr\"{o}der.
\newblock Admissible representations of probability measures.
\newblock {\em Electron. Notes Theor. Comput. Sci.}, 167:61--78, January 2007.

\bibitem[Sch09]{Schroder2009}
Matthias Schr\"{o}der.
\newblock An effective {Tietze}-{Urysohn} theorem for {QCB}-spaces.
\newblock {\em J. Univers. Comput. Sci.}, 15(6):1317--1336, 2009.

\bibitem[Shi95]{Shiryaev1995}
{Al'bert}~Nikolaevich Shiryaev.
\newblock {\em Probability}.
\newblock Springer, 1995.

\bibitem[Spi03]{Spitters2003}
Bas Spitters.
\newblock {\em Constructive and intuitionistic integration theory and
  functional analysis}.
\newblock PhD thesis, Katholieke Universiteit Nijmegen, 2003.

\bibitem[Spi06]{Spitters2006}
Bas Spitters.
\newblock Constructive algebraic integration theory.
\newblock {\em Ann. Pure Appl. Logic}, 137(1-3):380--390, 2006.

\bibitem[SS06]{SchroderSimpson2006}
Matthias Schr\"{o}der and Alex Simpson.
\newblock Probabilistic observations and valuations.
\newblock {\em Electron. Notes Theor. Comput. Sci.}, 155:605--615, May 2006.

\bibitem[Str72]{Street1972}
Ross Street.
\newblock The formal theory of monads.
\newblock {\em J. Pure Appl. Math.}, 2:149--168, 1972.

\bibitem[Tix95]{Tix1995}
R.~Tix.
\newblock {\em Stetige {B}ewertungen auf topologischen {R}{\"a}umen}.
\newblock PhD thesis, Master's Thesis, Technische Universit\"at Darmstadt,
  1995.

\bibitem[Var02]{Varacca2002}
Daniele Varacca.
\newblock The powerdomain of indexed valuations.
\newblock In {\em Proceedings of the 17th Annual IEEE Symposium on Logic in
  Computer Science}, pages 299--, Washington, DC, USA, 2002.

\bibitem[Vic08]{Vickers2008}
Steven Vickers.
\newblock A localic theory of lower and upper integrals.
\newblock {\em Math. Log. Quart.}, 54(1):109--123, 2008.

\bibitem[Vic11]{Vickers2011PP}
Steven Vickers.
\newblock A monad of valuation locales.
\newblock \url{http://www.cs.bham.ac.uk/~sjv/Riesz.pdf}, 2011.

\bibitem[Wei99]{Weihrauch1999}
Klaus Weihrauch.
\newblock Computability on the probability measures on the {Borel} sets of the
  unit interval.
\newblock {\em Theor. Comput. Sci.}, 219:421--437, May 1999.

\bibitem[WI81]{WatanabeIkeda1981}
S.~Watanabe and N.~Ikeda.
\newblock {\em Stochastic Differential Equations and Diffusion Processes}.
\newblock North-Holland, 1981.

\bibitem[YMT99]{YasugiMoriTsujii1999}
M.~Yasugi, T.~Mori, and Y.~Tsujii.
\newblock Effective properties of sets and functions in metric spaces with
  computability structure.
\newblock {\em Theor. Comput. Sci.}, 219(1-2):467--486, 1999.

\end{thebibliography}

\end{document}